\documentclass[12pt, reqno]{amsart}

\usepackage[T1]{fontenc}

\usepackage{amsmath, amssymb, amsfonts, amsthm}
\usepackage{mathtools}
\usepackage{latexsym}

\usepackage{amsbsy}

\usepackage[left=30mm, right=30mm, top=20mm, bottom=20mm]{geometry}

\usepackage{enumitem}
\usepackage{comment}

\usepackage{xcolor}

\usepackage{tikz}
\usetikzlibrary{circuits, intersections}
\usetikzlibrary{plotmarks}
\usetikzlibrary{angles, quotes}

\usepackage{graphicx}
\usepackage{pdfpages}

\usepackage[super]{nth}

\usepackage{hyperref}
\usepackage[colorinlistoftodos]{todonotes}

\newtheorem{thmintro}{Theorem}

\theoremstyle{definition}
\newtheorem{rkintro}{Remark}

\newcounter{countercheck}[section]

\theoremstyle{plain}
\newtheorem{theorem}{Theorem}[section]
\newtheorem{proposition}[theorem]{Proposition}
\newtheorem{lemma}[theorem]{Lemma}

\theoremstyle{definition}

\newtheorem{definition}[theorem]{Definition}
\newtheorem{convention}[theorem]{Convention}
\newtheorem{example}[theorem]{Example}

\theoremstyle{remark}
\newtheorem{remark}[theorem]{Remark}

\renewcommand{\phi}{\varphi}
\renewcommand{\epsilon}{\varepsilon}

\newcommand{\NN}{\mathbb{N}}
\newcommand{\ZZ}{\mathbb{Z}}

\newcommand{\RR}{\mathbb{R}}

\newcommand{\FF}{\mathbb{F}}

\newcommand{\KK}{\mathbb{K}}

\renewcommand{\P}{\mathcal{P}}

\DeclareMathOperator{\id}{id}

\DeclareMathOperator{\Aut}{Aut}

\DeclareMathOperator{\proj}{proj}

\numberwithin{equation}{section} 

\begin{document}

\title{Construction of Commutator Blueprints}

\author{Sebastian Bischof}

\thanks{email: bischof.math@icloud.com}

\thanks{UCLouvain, IRMP, Chemin du Cyclotron 2, 1348 Louvain-la-Neuve, Belgium}

\thanks{Keywords: Groups of Kac-Moody type, Commutation relations, Commutator blueprints}

\thanks{Mathematics Subject Classification 2020: 20D15, 20E42}

\begin{abstract}
	Commutator blueprints can be seen as blueprints for constructing RGD-systems over $\FF_2$ with prescribed commutation relations. In this paper we construct several families of Weyl-invariant commutator blueprints, mostly of universal type. Together with the main result of \cite{BiRGD} we obtain new examples of exotic RGD-systems of universal type over $\FF_2$.
\end{abstract}

\maketitle

\section{Introduction}

To any generalized Cartan matrix $A$ and any field $\FF$ we can construct an associated Kac-Moody group $\mathbf{G}_A(\FF)$ (cf.\ \cite{Ti87}). This group comes equipped with a distinguished family of subgroups $(U_{\alpha})_{\alpha \in \Phi}$ called \emph{root subgroups} indexed by the set $\Phi$ of real roots of $A$. The generalized Cartan matrix $A$ provides also a \emph{Weyl group} and hence a Coxeter system $(W(A), S(A))$. The real roots of $A$ correspond to the roots of $(W(A), S(A))$ (viewed as half-spaces). The pair $\left( \mathbf{G}_A(\FF), \left( U_{\alpha} \right)_{\alpha \in \Phi} \right)$ satisfies certain properties, for instance a commutation relation between root groups corresponding to prenilpotent pairs of roots. The combinatorics of such pairs had been formalized by Tits in the general framework of RGD-systems (cf.\ \cite{Ti92}).

Let $(W, S)$ be a Coxeter system and let $\Phi$ be its associated set of roots. An \emph{RGD-system of type $(W, S)$} is a pair $\mathcal{D} = \left( G, \left( U_{\alpha} \right)_{\alpha \in \Phi} \right)$ consisting of a group $G$ and a family of subgroup $\left( U_{\alpha} \right)_{\alpha \in \Phi}$ satisfying some axioms which are motivated by the theory of Kac-Moody groups. It is natural to ask whether any RGD-system is of \emph{Kac-Moody origin}. In the $2$-spherical case it follows from the main result of \cite{Mu99} that any RGD-system of irreducible and \emph{crystallographic} type (i.e.\ $o(st) \in \{2, 3, 4, 6\}$ for $s\neq t\in S$) having finite root groups of order $>3$ which generate the group $G$ must be a split or almost split Kac-Moody group in the sense of \cite{Re03}. In the right-angled case we have more flexibility. In \cite{RR06} Rémy and Ronan have constructed \emph{exotic} RGD-systems with prescribed isomorphism types of root groups. For mixed ground fields these groups are not of Kac-Moody origin. One main aspect in their construction is that all root groups corresponding to prenilpotent pairs of roots commute. The main goal of this paper is to establish the existence of new examples of exotic RGD-systems both of $2$-spherical and right-angled type having non-trivial commutation relations. In all examples, the root groups will have cardinality $2$.

In \cite{BiRGD} we have introduced the notion of a commutator blueprint; we refer to Section \ref{Section: commutator blueprints} for more details. These purely combinatorial objects can be seen as a blueprint for constructing RGD-systems with prescribed commutation relations. To each RGD-system one can associate a commutator blueprint. Such blueprints are called \emph{integrable}. A necessary condition for integrability is for the commutator blueprint to be \emph{Weyl-invariant} (roughly speaking:\ the commutation relations are Weyl-invariant). For this reason we will focus on constructing Weyl-invariant commutator blueprints.

\subsection*{Commutator blueprints of universal type}

Using the main result of \cite{BiRGD} one can show that Weyl-invariant commutator blueprints of \emph{universal} type (i.e.\ $o(st) = \infty$ for $s\neq t\in S$) are integrable (cf.\ \cite[Introduction]{BiRGD} for more details). This implies that the commutator blueprints which we will mention in Theorem \ref{Main theorem: Tits construction} -- \ref{Main theorem: U_w unbounded} below are integrable. The first family of examples generalizes Tits' construction of trivalent \emph{Moufang twin trees} in \cite[Section $5.4$ in $95/96$]{Ti73-00}, which are essentially the same as (center-free) RGD-systems of type $\tilde{A}_1$ (cf.\ Theorem \ref{Theorem: commutator blueprint}). In the case $\tilde{A}_1$ we have $\Phi = \{+, -\} \times \ZZ$ and all the examples constructed by Tits are of the following form, where $\epsilon \in \{+, -\}$:
\allowdisplaybreaks
\begin{align*}
	&&\forall z, z' \in \ZZ: &&[U_{{\epsilon, 2z}}, U_{\epsilon, z'}] = 1 &&\text{and} &&[U_{\epsilon, 2z+1}, U_{\epsilon, 2z'+1}] \leq \langle U_{\epsilon, 2i} \mid i \in \ZZ \rangle.
\end{align*}

\begin{thmintro}\label{Main theorem: Tits construction}
	Suppose $(W, S)$ is universal and of rank at least $2$. Then there exists a Weyl-invariant commutator blueprint $\mathcal{M}$ and $s\neq t\in S$ such that the restriction of $\mathcal{M}$ to the $\{s, t\}$-residue coincides with the commutator blueprint coming from Tits' construction of trivalent Moufang twin trees.
\end{thmintro}

\begin{rkintro}
	Suppose that $(W, S)$ has rank $2$. Then Theorem \ref{Main theorem: Tits construction} is just the commutator blueprint obtained from Tits' construction in \cite[Section $5.4$ in $95/96$]{Ti73-00}. Unfortunately, there is no proof available in the literature on the existence of these Moufang twin trees. Theorem \ref{Main theorem: Tits construction} together with \cite[Theorem A]{BiRGD} provides a different approach to the results about the trivalent Moufang twin trees constructed by Tits.
\end{rkintro}

\begin{rkintro}
	Again, we suppose that $(W, S)$ has rank $2$. The RGD-systems corresponding to the commutator blueprints in Theorem \ref{Main theorem: Tits construction} yield Moufang twin trees of order $2$ and the precise commutation relations can be found in Theorem \ref{Theorem: commutator blueprint}. The group $G$ from the RGD-systems -- which coincides with the automorphism of the corresponding trivalent Moufang twin tree -- is virtually simple in many cases. To see this, we choose in Theorem \ref{Theorem: commutator blueprint} $K$ to be finite and $J_k \neq \emptyset$ for some $k\in K$. Then the automorphism group of these Moufang twin trees is virtually simple by \cite[Theorem 1]{CR16}. If we assume additionally $1\in K$ and $J_1 := \{1\}$, then we can apply \cite[Lemma 6(i')]{CR16}. This guarantees that the commutator subgroup of the automorphism group of the Moufang twin tree associated with the RGD-system of type $\tilde{A}_1$ is simple and has finite index in the automorphism group. We note that there are elaborated results available which show that the automorphism group of almost all Moufang twin trees of order $2$ is virtually simple (cf.\ \cite{GHM16preprint2}). The proof is based on the idea given in \cite{CR16}.
\end{rkintro}

\begin{rkintro}
	Gr\"uninger, Horn and M\"uhlherr have shown in \cite{GHM16} that in the case $\tilde{A}_1$ over $\FF_2$ the commutation relations are generally very restrictive and that they cannot be significantly more complicated than those constructed by Tits. Thus it is natural to ask whether the commutation relations of any trivalent Moufang twin tree are as stated just before Theorem \ref{Main theorem: Tits construction}. In \cite[Introduction]{GHM16}, Gr\"uninger, Horn and M\"uhlherr announced that they constructed new examples of Moufang twin trees of order $2$ having different commutation relations than those constructed by Tits. In the present paper we provide independently an example with different commutation relations (cf.\ Theorem \ref{Theorem: exotic commutator blueprint} for the precise commutation relations). 
\end{rkintro}

We now turn our attention to the nilpotency class of the groups $U_w$. These groups appear as subgroups of RGD-systems and are generated by suitable root subgroups. It is a consequence of \cite[Theorem A]{GHM16} that in the case $\tilde{A}_1$ the group $U_w$ is nilpotent of class at most $2$ for all $w\in W$, provided that all root groups are isomorphic to $(\FF_p, +)$ for a fixed prime $p$. This result was generalized in \cite{PaDiss21} and includes the cases where $U_{\alpha_s} \cong (\KK_s, +)$ with $\KK_s$ a field of characteristic different from $2$ (cf.\ also \cite{SW08}, \cite{Seg09}).

In \cite[Theorem $1.2$]{Ca07} Caprace has shown that the nilpotency class of the groups $U_w$ in Kac-Moody groups is bounded above by a constant only depending on the generalized Cartan matrix $A$ and not on $w$. We will see that the general situation is very different and the results about Kac-Moody groups do not generalize to arbitrary RGD-systems. Even more, we can construct examples of any rank at least $3$ such that the nilpotency class of the groups $U_w$ can be arbitrarily large (cf.\ Theorem \ref{Theorem: M_nil(n)}\ref{Nilpotency class M_nil(n) bounded} and \cite[Theorem A]{BiRGD}). To make the statement precise, for an RGD-system $\mathcal{D}$ we define $\mathrm{ndeg}(\mathcal{D})$ to be the supremum of the nilpotency classes of the subgroups $U_w$ for all $w\in W$.

\begin{thmintro}\label{Main theorem: n-nilpotent}
	For each $n\in \NN_{\geq 3}$ there exists an RGD-system $\mathcal{D}_n$ with $\mathrm{ndeg}(\mathcal{D}_n) = n-1$.
\end{thmintro}

Theorem \ref{Main theorem: n-nilpotent} implies that a generalization of \cite[Theorem A]{GHM16} to higher rank is not possible. At present it is unclear whether for a fixed RGD-system $\mathcal{D}$, the nilpotency class of the groups $U_w$ is bounded above by a constant depending only on $\mathcal{D}$ and not on $w$. The next theorem -- which is a consequence of Theorem \ref{Theorem: M_nil(n)}\ref{Nilpotency class M_nil(n) unbounded} together with \cite[Theorem A]{BiRGD} -- shows that such a constant does not exist in general.

\begin{thmintro}\label{Main theorem: U_w unbounded}
	There exists an RGD-system $\mathcal{D}$ of each rank at least $3$ with $\mathrm{ndeg}(\mathcal{D}) = \infty$.
\end{thmintro}

\subsection*{Commutator blueprints of type $\mathbf{(4, 4, 4)}$}

So far, all commutator blueprints mentioned in the introduction are of universal type. We have also constructed commutator blueprints which are not of universal type. Suppose that $(W, S)$ is \emph{of type} $(4, 4, 4)$, that is, $(W, S)$ is of rank $3$ and $o(st) = 4$ for all $s\neq t\in S$. It is a consequence of \cite[Theorem A]{BiRGD} and \cite{BiDiss} that Weyl-invariant commutator blueprints of type $(4, 4, 4)$ are integrable and hence prescribe commutation relations in RGD-systems. This motivates the following theorem:

\begin{thmintro}
	There exist uncountably many Weyl-invariant commutator blueprints of type $(4, 4, 4)$.
\end{thmintro}

\subsection*{Overview}

In this paper we construct several families of Weyl-invariant commutator blueprints. In Section \ref{Section: Preliminaries} we fix notation and prove some elementary results about Coxeter systems. In Section \ref{Section: commutator blueprints} we recall the definition of commutator blueprints and we introduce the notion of a \emph{pre-commutator blueprint}. These objects are weaker versions of commutator blueprints which only depend on combinatorial properties. In Lemma \ref{Lemma: Conditions to extend U_w} we have worked out precise conditions on the commutation relations which guarantee $\vert U_w \vert = 2^{\ell(w)}$ and, in particular, that a pre-commutator blueprint is already a commutator blueprint. These conditions can be weakened and it turned out that these weaker conditions imply that the groups $U_w$ are nilpotent of class at most $2$ (cf.\ Theorem \ref{Theorem: pre-com = com}). Our construction of commutator blueprints (which is done in Section \ref{Section: Examples}) is always by constructing pre-commutator blueprints, which satisfy the conditions of either Lemma \ref{Lemma: Conditions to extend U_w} or of Theorem \ref{Theorem: pre-com = com}.

\subsection*{Acknowledgement}

I am very grateful to Timothée Marquis and Bernhard M\"uhlherr for many helpful discussions on the topic. This manuscript was partially written while the author was a PhD student at Justus-Liebig-Universität Gie\ss en. This work was partially supported by a fellowship of the German Academic Exchange Service (DAAD) via the grant 57664192.

\section{Preliminaries}\label{Section: Preliminaries}

\subsection*{Coxeter systems}

Let $(W, S)$ be a Coxeter system and let $\ell$ denote the corresponding length function. For $s, t \in S$ we denote the order of $st$ in $W$ by $m_{st}$. The \emph{Coxeter diagram} corresponding to $(W, S)$ is the labeled graph $(S, E(S))$, where $E(S) = \{ \{s, t \} \mid m_{st}>2 \}$ and where each edge $\{s,t\}$ is labeled by $m_{st}$ for all $s, t \in S$. The \emph{rank} of the Coxeter system is the cardinality of the set $S$. Let $(W, S)$ be of rank $3$ and let $S = \{ r, s, t \}$. Sometimes we will also call $(m_{rs}, m_{rt}, m_{st})$ the \emph{type} of $(W, S)$. If $3 \leq m_{rs}, m_{rt}, m_{st}$ and $(m_{rs}, m_{rt}, m_{st}) \neq (3, 3, 3)$, we call $(W, S)$ \emph{cyclic hyperbolic}.

It is well-known that for each $J \subseteq S$ the pair $(\langle J \rangle, J)$ is a Coxeter system (cf.\ \cite[Ch. IV, §$1$ Theorem $2$]{Bo68}). A subset $J \subseteq S$ is called \emph{spherical} if $\langle J \rangle$ is finite. The Coxeter system is called \emph{$2$-spherical} if $\langle J \rangle$ is finite for each $J \subseteq S$ with $\vert J \vert \leq 2$ (i.e.\ $m_{st} < \infty$); it is called \emph{spherical} if $S$ is spherical. Given a spherical subset $J$ of $S$, there exists a unique element of maximal length in $\langle J \rangle$, which we denote by $r_J$ (cf.\ \cite[Corollary $2.19$]{AB08}).

\subsection*{The chamber system $\mathbf{\Sigma(W, S)}$}

Let $(W, S)$ be a Coxeter system. Defining $w \sim_s w'$ if and only if $w^{-1}w' \in \langle s \rangle$ we obtain a chamber system with chamber set $W$ and equivalence relations $\sim_s$ for $s\in S$, which we denote by $\Sigma(W, S)$. We call two chambers $w, w'$ \textit{$s$-adjacent} if $w \sim_s w'$ and \textit{adjacent} if they are $s$-adjacent for some $s\in S$. A \textit{gallery of length $n$} from $w_0$ to $w_n$ is a sequence $(w_0, \ldots, w_n)$ of chambers where $w_i$ and $w_{i+1}$ are adjacent for each $0 \leq i < n$. A gallery $(w_0, \ldots, w_n)$ is called \textit{minimal} if there exists no gallery from $w_0$ to $w_n$ of length $k<n$ and we denote the length of a minimal gallery from $w_0$ to $w_n$ by $\ell(w_0, w_n)$. Let $G = (w_0, \ldots, w_n)$ be a minimal gallery. Then we call $(s_1, \ldots, s_n)$ the \emph{type} of $G$, where $ s_i := w_{i-1}^{-1} w_i \in S$.

For $J \subseteq S$ we define the \textit{$J$-residue} of a chamber $c\in W$ to be the set $R_J(c) := c \langle J \rangle$. A \textit{residue} $R$ is a $J$-residue for some $J \subseteq S$; we call $J$ the \textit{type} of $R$ and the cardinality of $J$ is called the \textit{rank} of $R$. A residue is called \emph{spherical} if its type is a spherical subset of $S$. Let $R$ be a spherical $J$-residue. Two chambers $x, y \in R$ are called \emph{opposite in $R$} if $\delta(x, y) = r_J$. Two residues $P, Q \subseteq R$ are called \emph{opposite in $R$} if for each $p\in P$ there exists $q\in Q$ such that $p, q$ are opposite in $R$. A \textit{panel} is a residue of rank $1$. It is a fact that for every chamber $x\in W$ and every residue $R$ there exists a unique chamber $z\in R$ such that $\ell(x, y) = \ell(x, z) + \ell(z, y)$ holds for every chamber $y\in R$. The chamber $z$ is called the \textit{projection} of $x$ onto $R$ and is denoted by $z = \proj_R x$.

\begin{example}
	The group $W$ acts on $\Sigma(W, S)$ by multiplication from the left (cf.\ \cite[$29.2$]{We09}).
\end{example}

A subset $\Sigma \subseteq W$ is called \emph{convex} if for any two chambers $c, d \in \Sigma$ and any minimal gallery $(c_0 = c, \ldots, c_k = d)$, we have $c_i \in \Sigma$ for all $0 \leq i \leq k$. Note that residues are convex by \cite[Example $5.44(b)$]{AB08}.

For two residues $R$ and $T$ we define $\proj_T R := \{ \proj_T r \mid r\in R \}$. By \cite[Lemma $5.36(2)$]{AB08} $\proj_T R$ is a residue contained in $T$. The residues $R$ and $T$ are called \emph{parallel} if $\proj_T R = T$ and $\proj_R T = R$.

\begin{lemma}\label{Lemma: parallel implies opposite}
	Let $R$ be a spherical residue of rank $2$ and let $P \neq Q$ be two parallel panels contained in $R$. Then $P$ and $Q$ are opposite in $R$.
\end{lemma}
\begin{proof}
	This is a consequence of \cite[Lemma $18$]{DMVM11} and \cite[Lemma $5.107$]{AB08}.
\end{proof}

\subsection*{Roots and walls}

Let $(W, S)$ be a Coxeter system. A \textit{reflection} is an element of $W$ that is conjugate to an element of $S$. For $s\in S$ we let $\alpha_s := \{ w\in W \mid \ell(sw) > \ell(w) \}$ be the \textit{simple root} corresponding to $s$. A \textit{root} is a subset $\alpha \subseteq W$ such that $\alpha = v\alpha_s$ for some $v\in W$ and $s\in S$. We denote the set of all roots by $\Phi(W, S)$. We note that roots are convex (cf.\ \cite[Proposition $5.81$]{AB08}). The set $\Phi(W, S)_+ := \{ \alpha \in \Phi(W, S) \mid 1_W \in \alpha \}$ is the set of all \textit{positive roots} and $\Phi(W, S)_- := \{ \alpha \in \Phi(W, S) \mid 1_W \notin \alpha \}$ is the set of all \textit{negative roots}. For each root $\alpha \in \Phi(W, S)$ we denote its \textit{opposite root} by $-\alpha$ and we denote the unique reflection which interchanges these two roots by $r_{\alpha}$. For $\alpha \in \Phi(W, S)$ we denote by $\partial \alpha$ (resp.\ $\partial^2 \alpha$) the set of all panels (resp.\ spherical residues of rank $2$) stabilized by $r_{\alpha}$. Furthermore, we define $\mathcal{C}(\partial \alpha) := \bigcup_{P \in \partial \alpha} P$ and $\mathcal{C}(\partial^2 \alpha) := \bigcup_{R \in \partial^2 \alpha} R$. The set $\partial \alpha$ is called the \textit{wall} associated with $\alpha$. Let $G = (c_0, \ldots, c_k)$ be a gallery. We say that $G$ \textit{crosses the wall $\partial \alpha$} if there exists $1 \leq i \leq k$ such that $\{ c_{i-1}, c_i \} \in \partial \alpha$. It is a basic fact that a minimal gallery crosses a wall at most once (cf.\ \cite[Lemma $3.69$]{AB08}).

\begin{convention}
	For the rest of this paper we let $(W, S)$ be a Coxeter system of finite rank and we define $\Phi := \Phi(W, S)$ (resp.\ $\Phi_+ := \Phi(W, S)_+$ and $\Phi_- := \Phi(W, S)_-$). Moreover, we assume that $m_{st} \in \{2, 3, 4, 6, 8, \infty\}$ for all $s\neq t \in S$.
\end{convention}

A pair $\{ \alpha, \beta \} \subseteq \Phi$ of roots is called \emph{prenilpotent}, if $\alpha \cap \beta \neq \emptyset \neq (-\alpha) \cap (-\beta)$. For a prenilpotent pair $\{ \alpha, \beta \}$ of roots we will write $\left[ \alpha, \beta \right] := \{ \gamma \in \Phi \mid \alpha \cap \beta \subseteq \gamma \text{ and } (-\alpha) \cap (-\beta) \subseteq (-\gamma) \}$ and $(\alpha, \beta) := \left[ \alpha, \beta \right] \backslash \{ \alpha, \beta \}$. A pair $\{ \alpha, \beta \}$ of roots is called \emph{nested} if $\alpha \subseteq \beta$ or $\beta \subseteq \alpha$.

Let $(c_0, \ldots, c_k)$ and $(d_0 = c_0, \ldots, d_k = c_k)$ be two minimal galleries from $c_0$ to $c_k$ and let $\alpha \in \Phi$. Then $\partial \alpha$ is crossed by the minimal gallery $(c_0, \ldots, c_k)$ if and only if it is crossed by the minimal gallery $(d_0, \ldots, d_k)$. For $\alpha_1, \ldots, \alpha_k \in \Phi$ we say that a minimal gallery $G = (c_0, \ldots, c_k)$ \textit{crosses the sequence of roots} $(\alpha_1, \ldots, \alpha_k)$, if $c_{i-1} \in \alpha_i$ and $c_i \notin \alpha_i$ for all $1\leq i \leq k$. In this case we say that $G$ is a minimal gallery \emph{between} $\alpha_1$ and $\alpha_k$.

We denote the set of all minimal galleries $(c_0 = 1_W, \ldots, c_k)$ starting at $1_W$ by $\mathrm{Min}$. For $w\in W$ we denote the set of all $G \in \mathrm{Min}$ of type $(s_1, \ldots, s_k)$ with $w = s_1 \cdots s_k$ by $\mathrm{Min}(w)$. For $w\in W$ with $\ell(sw) = \ell(w) -1$ we let $\mathrm{Min}_s(w)$ be the set of all $G \in \mathrm{Min}(w)$ of type $(s, s_2, \ldots, s_k)$. We extend this notion to the case $\ell(sw) = \ell(w) +1$ by defining $\mathrm{Min}_s(w) := \mathrm{Min}(w)$. Let $w\in W, s\in S$ and $G = (c_0, \ldots, c_k) \in \mathrm{Min}_s(w)$. If $\ell(sw) = \ell(w) -1$, then $c_1 = s$ and we define $sG := (sc_1 = 1_W, \ldots, sc_k) \in \mathrm{Min}(sw)$. If $\ell(sw) = \ell(w) +1$, we define $sG := (1_W, sc_0 = s, \ldots, sc_k) \in \mathrm{Min}(sw)$.

\begin{lemma}\label{Lemma: residue and root}
	Let $R$ be a spherical residue of $\Sigma(W, S)$ of rank $2$ and let $\alpha \in \Phi$. Then exactly one of the following hold:
	\begin{enumerate}[label=(\alph*)]
		\item $R \subseteq \alpha$;
		
		\item $R \subseteq (-\alpha)$;
		
		\item $R \in \partial^2 \alpha$;
	\end{enumerate}
\end{lemma}
\begin{proof}
	This is \cite[Lemma $2.2$]{BiCoxGrowth}.
\end{proof}

\begin{lemma}\label{CM06Prop2.7}
	Let $R$ and $T$ be two spherical residues of $\Sigma(W, S)$. Then the following are equivalent
	\begin{enumerate}[label=(\roman*)]
		\item $R$ and $T$ are parallel;
		
		\item a reflection of $\Sigma(W, S)$ stabilizes $R$ if and only if it stabilizes $T$;
		
		\item there exist two sequences $R_0 = R, \ldots, R_n = T$ and $T_1, \ldots, T_n$ of residues of spherical type such that for each $1 \leq i \leq n$ the rank of $T_i$ is equal to $1+\mathrm{rank}(R)$, the residues $R_{i-1}, R_i$ are contained and opposite in $T_i$ and moreover, we have $\proj_{T_i} R = R_{i-1}$ and $\proj_{T_i} T = R_i$.
	\end{enumerate}
\end{lemma}
\begin{proof}
	This is \cite[Proposition $2.7$]{CM06}.
\end{proof}

\begin{lemma}\label{CM05Lem2.3}
	Let $\alpha \in \Phi$ be a root and let $x, y \in \alpha \cap \mathcal{C}(\partial \alpha)$. Then there exists a minimal gallery $(c_0 = x, \ldots, c_k = y)$ such that $c_i \in \mathcal{C}(\partial^2 \alpha)$ for each $0 \leq i \leq k$. Moreover, for every $1 \leq i \leq k$ there exists $L_i \in \partial^2 \alpha$ with $\{ c_{i-1}, c_i \} \subseteq L_i$.
\end{lemma}
\begin{proof}
	This is a consequence of \cite[Lemma $2.3$]{CM05} and its proof.
\end{proof}

\begin{lemma}\label{Lemma: residues in the 2-boundary are parallel}
	Let $\alpha, \beta \in \Phi, \alpha \neq \pm \beta$ be two roots and let $R, T \in \partial^2 \alpha \cap \partial^2 \beta$.
	\begin{enumerate}[label=(\alph*)]
		\item The residues $R$ and $T$ are parallel.
		
		\item If $\vert \langle J \rangle \vert = \infty$ holds for all $J \subseteq S$ containing three elements, then $R=T$.
	\end{enumerate}
\end{lemma}
\begin{proof}
	This is \cite[Lemma $2.5$]{BiCoxGrowth}.
\end{proof}

\begin{lemma}\label{Lemma: nested infinite order}
	Let $\alpha, \beta \in \Phi, \alpha \neq \pm \beta$, be two roots.
	\begin{enumerate}[label=(\alph*)]
		\item The following are equivalent:
		\begin{enumerate}[label=(\roman*)]
			\item Either $\{ \alpha, \beta \}$ is nested or $\{ -\alpha, \beta \}$ is nested.
			
			\item We have $o(r_{\alpha} r_{\beta}) = \infty$.
			
			\item We have $\partial^2 \alpha \cap \partial^2 \beta = \emptyset$.
		\end{enumerate}
		Moreover, if $o(r_{\alpha} r_{\beta}) = \infty$ and $\{ \alpha, \beta \}$ is a prenilpotent pair, then $\{ \alpha, \beta \}$ is nested.
		
		\item If $o(r_{\alpha} r_{\beta}) < \infty$ and $\gamma \in (\alpha, \beta)$, then $\partial^2 \alpha \cap \partial^2 \beta \cap \partial^2 \gamma \neq \emptyset$ and $o(r_{\alpha} r_{\gamma}), o(r_{\beta} r_{\gamma}) < \infty$.
	\end{enumerate}
\end{lemma}
\begin{proof}
	The implication $(i) \Rightarrow (ii)$ follows exactly as in \cite[Proposition $3.165$]{AB08}. Now suppose $(ii)$ and assume that there exists $R \in \partial^2 \alpha \cap \partial^2 \beta$. As $R$ is finite, there exists $k \in \NN$ such that $(r_{\alpha} r_{\beta})^k$ fixes a chamber in $R$, i.e.\ $(r_{\alpha} r_{\beta})^k w = (r_{\alpha} r_{\beta})^k(w) = w$ for some $w\in R$. But this implies $(r_{\alpha} r_{\beta})^k = 1$. As $o(r_{\alpha} r_{\beta}) = \infty$, we obtain a contradiction. Now suppose that non of $\{ \alpha, \beta \}, \{ -\alpha, \beta \}$ is nested. Then we have $\alpha \not\subseteq \beta, \beta \not\subseteq \alpha$ and $(-\alpha) \not\subseteq \beta, \beta \not\subseteq (-\alpha)$. This implies that non of $\alpha \cap (-\beta), \beta \cap (-\alpha), (-\alpha) \cap (-\beta), \beta \cap \alpha$ is the empty set. Arguing as in the proof of \cite[Proposition $29.24$]{We09} and using Lemma \ref{CM06Prop2.7}, there exists $R \in \partial^2 \alpha \cap \partial^2 \beta$ and we are done. The second part of $(a)$ follows from $(a)$ and \cite[Lemma $8.42(c)$]{AB08}.
	
	To show $(b)$ we note that by $(a)$ there exists $R \in \partial^2 \alpha \cap \partial^2 \beta$. We deduce $\emptyset \neq R \cap \alpha \cap \beta \subseteq \gamma$ and $\emptyset \neq R \cap (-\alpha) \cap (-\beta) \subseteq (-\gamma)$. If follows from Lemma \ref{Lemma: residue and root} that $R \in \partial^2 \gamma$. In particular, $R \in \partial^2 \alpha \cap \partial^2 \beta \cap \partial^2 \gamma$. We deduce $o(r_{\alpha} r_{\gamma}) < \infty$ and $o(r_{\beta} r_{\gamma}) < \infty$ from $(a)$.
\end{proof}

\begin{lemma}\label{Lemma: residue cut by intersection reflections}
	Let $\alpha, \beta, \gamma \in \Phi$ be three pairwise distinct and pairwise non-opposite roots such that $\partial^2 \alpha \cap \partial^2 \beta \cap \partial^2 \gamma \neq \emptyset$ (e.g.\ $\alpha \neq \pm \beta, o(r_{\alpha} r_{\beta}) < \infty,\gamma \in (\alpha, \beta)$). Then the following hold:
	\begin{enumerate}[label=(\alph*)]
		\item $\partial^2 \alpha \cap \partial^2 \beta = \partial^2 \alpha \cap \partial^2 \gamma$;
		
		\item $\left( (\alpha, \beta) \cup (-\alpha, \beta) \right) \cap \{ \gamma, -\gamma \} \neq \emptyset$.	
	\end{enumerate}
\end{lemma}
\begin{proof}
	Let $R \in \partial^2 \alpha \cap \partial^2 \beta \cap \partial^2 \gamma$ be a residue, let $\delta \in \{ \beta, \gamma \}$ and let $T \in \partial^2 \alpha \cap \partial^2 \delta$. It suffices to show that $T \in \partial^2 \alpha \cap \partial^2 \beta \cap \partial^2 \gamma$. Using Lemma \ref{Lemma: residues in the 2-boundary are parallel}, we deduce that $R$ and $T$ are parallel. Then Lemma \ref{CM06Prop2.7} implies that a reflection of $\Sigma(W, S)$ stabilizes $R$ if and only if it stabilizes $T$. As $r_{\alpha}, r_{\beta}, r_{\gamma}$ stabilize $R$, they also stabilize $T$ and assertion $(a)$ follows. Before we show $(b)$ we prove the following claim:
	
	\emph{Claim:} $\gamma \notin (\alpha, \beta) \Rightarrow \alpha \cap \beta \cap (-\gamma) \cap R \neq \emptyset$.
	
	We have $\alpha \cap \beta \not\subseteq \gamma$ or $(-\alpha) \cap (-\beta) \not\subseteq (-\gamma)$. In the first case we have $\alpha \cap \beta \cap (-\gamma) \neq \emptyset$ and, as roots are convex, the claim follows from \cite[Lemma $5.45$]{AB08}. Thus we can assume $(-\alpha) \cap (-\beta) \not\subseteq (-\gamma)$. As roots are convex, \cite[Lemma $5.45$]{AB08} implies $(-\alpha) \cap (-\beta) \cap \gamma \cap R \neq \emptyset$. Let $x$ be contained in this set and let $y \in R$ be opposite to $x$ in $R$. Then $y \in \alpha \cap \beta \cap (-\gamma) \cap R$ and the claim follows.
	
	We are now in the position to prove $(b)$. We assume $(\alpha, \beta) \cap \{ \gamma, -\gamma \} = \emptyset = (-\alpha, \beta) \cap \{ \gamma, -\gamma \}$. By the above we deduce the following:
	\allowdisplaybreaks
	\begin{align*}
		&x \in \alpha \cap \beta \cap (-\gamma) \cap R &&\text{and} &&x' \in \alpha \cap \beta \cap \gamma \cap R, \\
		&y \in (-\alpha) \cap \beta \cap (-\gamma) \cap R &&\text{and} &&y' \in (-\alpha) \cap \beta \cap \gamma \cap R
	\end{align*}
	As residues and roots are convex, there exist $P, Q \in \partial \gamma$ such that $P \subseteq \alpha \cap \beta \cap R$ and $Q \subseteq (-\alpha) \cap \beta \cap R$. As $P \subseteq \alpha$ and $Q \subseteq (-\alpha)$, we have $P \neq Q$ and Lemma \ref{Lemma: parallel implies opposite} implies that there exist $p\in P, q\in Q$ which are opposite in $R$. Using \cite[Proposition $5.4$]{We03}, every chamber in $R$ lies on a minimal gallery from $p$ to $q$. As roots are convex and $p, q \in \beta$, we infer $R \subseteq \beta$, which is a contradiction to $R \in \partial^2 \beta$.
\end{proof}

\begin{lemma}\label{Lemma: interval of roots}
	Suppose $\alpha \neq \pm\beta \in \Phi$ with $o(r_{\alpha} r_{\beta}) < \infty$ and let $\gamma \in (\alpha, \beta)$. Then $\{ -\alpha, \gamma \}$ is a prenilpotent pair and we have $\beta \in (-\alpha, \gamma)$. 
\end{lemma}
\begin{proof}
	Note that by Lemma \ref{Lemma: nested infinite order} and \cite[Lemma $8.42(3)$]{AB08} the pair $\{ -\alpha, \gamma \}$ is prenilpotent. Thus we have to show that $(-\alpha) \cap \gamma \subseteq \beta$ and $\alpha \cap (-\gamma) \subseteq (-\beta)$. As $\gamma \in (\alpha, \beta)$, we have $\alpha \cap \beta \subseteq \gamma$ and $(-\alpha) \cap (-\beta) \subseteq (-\gamma)$. In particular, we have $(-\gamma) \subseteq (-\alpha) \cup (-\beta)$ and $\gamma \subseteq \alpha \cup \beta$. We compute:
	\allowdisplaybreaks
	\begin{align*}
		(-\alpha) \cap \gamma &\subseteq (-\alpha) \cap (\alpha \cup \beta) \subseteq (-\alpha) \cap \beta \subseteq \beta \\
		\alpha \cap (-\gamma) &\subseteq \alpha \cap ((-\alpha) \cup (-\beta)) \subseteq \alpha \cap (-\beta) \subseteq (-\beta) \qedhere
	\end{align*}
\end{proof}

\subsection*{Reflection and combinatorial triangles in $\mathbf{\Sigma(W, S)}$}

A \textit{reflection triangle} is a set $R$ of three reflections such that the order of $tt'$ is finite for all $t, t' \in R$ and such that $\bigcap_{t\in R} \partial^2 \beta_t = \emptyset$, where $\beta_t$ is one of the two roots associated with the reflection $t$. Note that $\partial^2 \beta_t = \partial^2 (-\beta_t)$. A set of three roots $T$ is called \textit{combinatorial triangle} (or simply \textit{triangle}) if the following hold:

\begin{enumerate}[label=(CT\arabic*)]
	\item The set $\{ r_{\alpha} \mid \alpha \in T \}$ is a reflection triangle.
	
	\item For each $\alpha \in T$, there exists $\sigma \in \partial^2 \beta \cap \partial^2 \gamma$ with $\sigma \subseteq \alpha$, where $\{ \beta, \gamma \} = T \backslash \{ \alpha \}$.
\end{enumerate}

\begin{remark}\label{Remark: reflection triangle plus orientation yields triangle}
	Let $R$ be a reflection triangle. Then there exist three roots $\beta_1, \beta_2, \beta_3 \in \Phi$ such that $R = \{ r_{\beta_1}, r_{\beta_2}, r_{\beta_3} \}$. Let $\{i, j, k\} = \{1, 2, 3\}$. As $o(r_{\beta_i} r_{\beta_j}) < \infty$, there exists $\sigma_k \in \partial^2 \beta_i \cap \partial^2 \beta_j$ by Lemma \ref{Lemma: nested infinite order}. Since $R$ is a reflection triangle, we have $\sigma_k \notin \partial^2 \beta_k$ and Lemma \ref{Lemma: residue and root} yields $\sigma_k \subseteq \beta_k$ or $\sigma_k \subseteq -\beta_k$. Let $\epsilon_k \in \{+, -\}$ with $\sigma_k \subseteq \epsilon_k \beta_k$ and define $\alpha_k := \epsilon_k \beta_k$. Then $\{ \alpha_1, \alpha_2, \alpha_3 \}$ is a triangle, which induces the reflection triangle $R$.
\end{remark}

\begin{lemma}\label{reflectiontrianglechamber}
	Suppose that $(W, S)$ is $2$-spherical and let $T$ be a triangle.
	\begin{enumerate}[label=(\alph*)]
		\item If the Coxeter diagram is the complete graph, then $(-\alpha, \beta) = \emptyset$ for all $\alpha \neq \beta \in T$.
		
		\item If $(W, S)$ is cyclic hyperbolic, then $T$ contains a unique chamber, i.e.\ $\vert \bigcap_{\alpha \in T} \alpha \vert = 1$.
	\end{enumerate}
\end{lemma}
\begin{proof}
	Assertion $(a)$ is \cite[Proposition $2.3$]{Bi22} and assertion $(b)$ is a consequence of the classification in \cite{Fe98} (cf.\ Figure $8$ in $\S 5.1$ in loc.\ cit.).
\end{proof}

\begin{lemma}\label{Lemma: Triangle and infinite order}
	Suppose that $(W, S)$ is $2$-spherical and that the Coxeter diagram is the complete graph. Let $\{ \alpha_1, \alpha_2, \alpha_3 \}$ be a triangle and let $\beta \in (\alpha_1, \alpha_2)$.
	\begin{enumerate}[label=(\alph*)]
		\item We have $o(r_{\beta} r_{\alpha_3}) = \infty$ and $-\alpha_3 \subseteq \beta$.
		
		\item If $(W, S)$ is cyclic hyperbolic, then $(-\alpha_3, \beta) = \emptyset$.
	\end{enumerate}
\end{lemma}
\begin{proof}
	Let $i \in \{1, 2\}$ and assume that $o(r_{\alpha_3} r_{\beta}) < \infty$. Then $\{ r_{\alpha_i}, r_{\beta}, r_{\alpha_3} \}$ is a reflection triangle by Lemma \ref{Lemma: nested infinite order}$(b)$ and \ref{Lemma: residue cut by intersection reflections}$(a)$. We will apply Remark \ref{Remark: reflection triangle plus orientation yields triangle} to determine the triangle which induces $\{ r_{\alpha_1}, r_{\beta}, r_{\alpha_3} \}$. We note that $\alpha_2 \in (-\alpha_1, \beta)$ and $\alpha_1 \in (-\alpha_2, \beta)$ hold by Lemma \ref{Lemma: interval of roots}.
	
	There exist $R \in \partial^2 \alpha_1 \cap \partial^2 \alpha_2 = \partial^2 \alpha_1 \cap \partial^2 \beta$ with $R \subseteq \alpha_3$ and $R' \in \partial^2 \alpha_1 \cap \partial^2 \alpha_3$ with $R' \subseteq \alpha_2$. Thus $\emptyset \neq R' \cap \alpha_1 \subseteq \alpha_1 \cap \alpha_2 \subseteq \beta$ and Lemma \ref{Lemma: residue and root} yields $R \subseteq \beta$. This implies that there exists $\epsilon \in \{+, -\}$ such that $\{ \epsilon \alpha_1, \beta, \alpha_3 \}$ is a triangle. If $\epsilon = +$, then $(-\alpha_1, \beta) = \emptyset$ by Lemma \ref{reflectiontrianglechamber}. But this is a contradiction to the fact that $\alpha_2 \in (-\alpha_1, \beta)$. Thus we have $\epsilon = -$. We show that $\{ \beta, \alpha_2, \alpha_3 \}$ is a triangle.
	
	There exists $R_3 \in \partial^2 \alpha_1 \cap \partial^2 \alpha_2 = \partial^2 \beta \cap \partial^2 \alpha_2$ with $R_3 \subseteq \alpha_3$. There exists $R_1 \in \partial^2 \alpha_2 \cap \partial^2 \alpha_3$ with $R_1 \subseteq \alpha_1$. In particular, $\emptyset \neq R_1 \cap \alpha_2 \subseteq \alpha_1 \cap \alpha_2 \subseteq \beta$. Now, as $\{ -\alpha_1, \beta, \alpha_3 \}$ is a triangle, there exists $T \in \partial^2 \beta \cap \partial^2 \alpha_3$ with $T \subseteq (-\alpha_1)$. Thus $\emptyset \neq \beta \cap T \subseteq \beta \cap (-\alpha_1) \subseteq \alpha_2$. We infer that $\{ \beta, \alpha_2, \alpha_3 \}$ is a triangle and hence $(-\alpha_2, \beta) = \emptyset$ by Lemma \ref{reflectiontrianglechamber}. But this is again a contradiction. We conclude $o(r_{\alpha_3} r_{\beta}) = \infty$.
	
	Note that $\emptyset \neq R_3 \cap (-\beta) \subseteq \alpha_3$ and $\emptyset \neq R_1 \cap \alpha_2 \cap (-\alpha_3) \subseteq \alpha_1 \cap \alpha_2 \subseteq \beta$. In particular, $\{-\alpha_3, \beta\}$ is a prenilpotent pair and by Lemma \ref{Lemma: nested infinite order} it is nested. We have $\emptyset \neq R_3 \cap \beta \subseteq \alpha_3$. But this implies $\beta \not\subseteq (-\alpha_3)$. We deduce $(-\alpha_3) \subseteq \beta$. This finishes the proof of $(a)$.
	
	Let $\{x, y\} \in \partial \alpha_3$ such that $\bigcap_{i=1}^3 \alpha_i = \{y\}$ (cf.\ Lemma \ref{reflectiontrianglechamber}) and let $R \in \partial^2 \alpha_1 \cap \partial^2 \alpha_2$. By Lemma \ref{Lemma: residues in the 2-boundary are parallel}$(b)$ $R$ is unique and hence $y\in R$. Let $d\in R$ be opposite to $y$ in $R$ and let $(c_0 = x, c_1 = y, \ldots, c_n = d)$ be a minimal gallery. Then $c_i \in R$ for each $1 \leq i \leq n$. Let $(\beta_1, \ldots, \beta_n)$ be the sequence of roots crossed by $(c_0, \ldots, c_n)$. Then $\beta_1 = -\alpha_3$ and $o(r_{\beta_i} r_{\beta}) < \infty$ for each $2 \leq i \leq n$ by Lemma \ref{Lemma: nested infinite order}. Assume $(-\alpha_3, \beta) \neq \emptyset$. Then \cite[Lemma $3.69$]{AB08} implies that for each $\gamma \in (-\alpha_3, \beta)$ there exists $2 \leq i \leq n-1$ with $\gamma = \beta_i$. As $\gamma \subsetneq \beta$, this is a contradiction and hence $(-\alpha_3, \beta) = \emptyset$.
\end{proof}

\begin{lemma}\label{wordsincoxetergroup}
	Assume that $(W, S)$ is $2$-spherical and that the Coxeter diagram is the complete graph. Suppose $w\in W$ and $s\neq t \in S$ with $\ell(ws) = \ell(w) +1 = \ell(wt)$ and suppose $w' \in \langle s, t\rangle$ with $\ell(w') \geq 2$. Then we have $\ell(ww'r) = \ell(w) + \ell(w') +1$ for each $r\in S\backslash \{s, t\}$.
	
	If, moreover, $m_{pq} \neq 3$ for all $p, q \in S$, then we have $\ell(ww'rf) = \ell(w) + \ell(w') +2$ for each $r\in S \backslash \{s, t\}$ and $f\in S \backslash \{ r \}$.
\end{lemma}
\begin{proof}
	The first part is \cite[Corollary $2.8$]{BiCoxGrowth} and the second part is a consequence of the first.
\end{proof}

\begin{lemma}\label{projRiinduction}
	Assume that $(W, S)$ is $2$-spherical and that the Coxeter diagram is the complete graph. Let $\alpha \in \Phi_+$ be a root and let $P, Q \in \partial \alpha$. Let $P_0 = P, \ldots, P_n = Q$ and $R_1, \ldots, R_n$ be as in Lemma \ref{CM06Prop2.7}. If $\proj_{R_i} 1_W = \proj_{P_{i-1}} 1_W$ for some $1 \leq i \leq n$, then $\proj_{R_n} 1_W = \proj_{P_{n-1}}1_W$.
\end{lemma}
\begin{proof}
	We will show the hypothesis by induction on $n-i$. If $n-i = 0$, then there is nothing to show. Thus we suppose $n-i >0$. Let $J_i$ be the type of $R_i$ and let $w := \proj_{R_i} 1_W = \proj_{P_{i-1}} 1_W \in P_{i-1}$. As $P_{i-1} \neq P_i$ are contained and opposite in $R_i$ by Lemma \ref{CM06Prop2.7}, there exists $w' \in P_i$ such that $w$ and $w'$ are opposite in $R_i$, i.e.\ $w' = wr_{J_i}$. For $t\in S$ with $J_i \cap J_{i+1} = \{t\}$ we have $R_i \cap R_{i+1} = P_i = \P_t(w') = \P_t(wr_{J_i})$. As $w = \proj_{R_i} 1_W$, we deduce $\ell( \proj_{P_i} 1_W ) = \ell(w (r_{J_i} t)) \geq \ell(w) +2$. Using Lemma \ref{wordsincoxetergroup}, we infer $\ell((\proj_{P_i} 1_W)s) = \ell(\proj_{P_i} 1_W) +1$ for $s \in J_{i+1} \backslash J_i$ and hence $\proj_{R_{i+1}} 1_W = \proj_{P_i} 1_W$. Using induction the claim follows.
\end{proof}

\begin{lemma}\label{Lemma: residues in pairs}
	Assume that $(W, S)$ is $2$-spherical and that the Coxeter diagram is the complete graph. Let $w\in W$ and $r, s, t\in S$ be pairwise distinct such that $\ell(ws) = \ell(w) +1 = \ell(wt)$. If $\ell(wsr) = \ell(w)$, then $\ell(wsrt) = \ell(w) +1$.
\end{lemma}
\begin{proof}
	This is a consequence of Lemma \ref{wordsincoxetergroup}.
\end{proof}

\begin{lemma}\label{Lemma: not both down}
	Assume that $(W, S)$ is $2$-spherical and that $m_{st} \geq 4$ for all $s\neq t \in S$. Suppose $w \in W$ and $s\neq t \in S$ with $\ell(ws) = \ell(w) +1 = \ell(wt)$. Then we have $\ell(w) +2 \in \{ \ell(wsr), \ell(wtr) \}$ for all $r\in S \backslash \{s, t\}$.
\end{lemma}
\begin{proof}
	This is \cite[Lemma $2.9$]{BiCoxGrowth}.
\end{proof}

\begin{lemma}\label{mingallinrep}
	Suppose that $m_{st} = 4$ for all $s\neq t\in S$. Let $r, s, t \in S$ be pairwise distinct and let $p \in S \backslash \{s, t\}$. Let $H = (d_0, \ldots, d_4)$ be a minimal gallery of type $(r, s, t, p)$ and let $\alpha \in \Phi$ with $\{ d_0, d_1 \} \in \partial \alpha$ and $d_0 \in \alpha$. Let $\beta$ be a root containing $d_0$ and $\{ \{d_2, d_3\}, \{d_3, d_4\} \} \cap \partial \beta \neq \emptyset$ Then we have $\alpha \subsetneq \beta$.
\end{lemma}
\begin{proof}
	We use the canonical linear representation of $(W, S)$ (cf.\ \cite[Ch.\ $2.5$]{AB08}). Let $V := \RR^S$ be the vector space over $\RR$ with standard basis $(e_s)_{s\in S}$ and let $(\cdot, \cdot)$ be the symmetric bilinear form on $V$ given by
	\allowdisplaybreaks
	\begin{align*}
		(e_s, e_t) := -\cos\left( \frac{\pi}{m_{st}} \right) = \begin{cases*}
			1 & \text{if } s=t, \\
			-\frac{\sqrt{2}}{2} & \text{else}.
		\end{cases*}
	\end{align*}
	Then $W$ acts on $V$ via $\sigma: W \to \mathrm{GL}(V), s \mapsto \left( \sigma_s: V \to V, x \mapsto x - 2(x, e_s) e_s \right)$ and $(\cdot, \cdot)$ is invariant under this action. Let $\alpha$ and $\beta$ be as in the statement. Without loss of generality we can assume $\alpha = \alpha_r$ and $\beta \in \{ rs\alpha_t, rst\alpha_p \}$. At first, we consider the case $\beta = rs\alpha_t$. We compute:
	\allowdisplaybreaks
	\begin{align*}
		(e_r, \sigma(rs)(e_t)) &= ( e_r, \sigma_r( \sigma_s(e_t) ) ) = ( \sigma_r(e_r), \sigma_s(e_t) ) = (-e_r, e_t + \sqrt{2} e_s) = \frac{\sqrt{2}}{2} +1 >1
	\end{align*}
	Now we assume $\beta = rst\alpha_p$ and we compute:
	\allowdisplaybreaks
	\begin{align*}
		(e_r, \sigma(rst)(e_p)) &= (e_r, \sigma_r(\sigma_s(\sigma_t(e_p)))) \\
		&= ( \sigma_s(-e_r), \sigma_t(e_p) ) \\
		&= (-e_r -2(-e_r, e_s)e_s, e_p -2(e_p, e_t)e_t) \\
		&= -(e_r, e_p) + 2(e_p, e_t)(e_r, e_t) + 2(e_r, e_s)(e_s, e_p) -4(e_r, e_s)(e_p, e_t)(e_s, e_t) \\
		&= -(e_r, e_p) +1 +1 +\sqrt{2} >1
	\end{align*}
	Using \cite[Lemma $2.77$]{AB08} we obtain that $o(r_{\alpha} r_{\beta}) = \infty$. As $\{ \alpha, \beta \}$ is a prenilpotent pair, Lemma \ref{Lemma: nested infinite order} yields that $\{ \alpha, \beta \}$ is a pair of nested roots and hence $\alpha \subsetneq \beta$.
\end{proof}

\subsection*{Roots in Coxeter systems}

\begin{convention}
	For the rest of this section we assume that $(W, S)$ is $2$-spherical and $m_{st} \geq 4$ for all $s\neq t \in S$.
\end{convention}

Let $R$ be a residue and let $\alpha \in \Phi_+$. Then we call $\alpha$ a \textit{simple root of $R$} if there exists $P \in \partial \alpha$ such that $P \subseteq R$ and $\proj_R 1_W = \proj_P 1_W$. In this case $R$ is also stabilized by $r_{\alpha}$ and hence $R \in \partial^2 \alpha$. For a positive root $\alpha \in \Phi_+$ we define $k_{\alpha} := \min \{ k\in \NN \mid \exists\, (c_0, \ldots, c_k) \in \mathrm{Min}: c_k \notin \alpha \}$. We remark that $k_{\alpha} = 1$ if and only if $\alpha$ is a simple root.

\begin{remark}\label{Remark: any non-simple root is non-simple root of some residue}
	Let $\alpha \in \Phi_+$ be a positive root with $k_{\alpha} >1$. Let $(c_0, \ldots, c_k) \in \mathrm{Min}$ be a minimal gallery with $k = k_{\alpha}$ and $c_k \notin \alpha$. Then $\{ c_{k-1}, c_k \} \in \partial \alpha$ and $\alpha$ is not a simple root of the rank $2$ residue containing $c_{k-2}, c_{k-1}, c_k$. In particular, there exists $R \in \partial^2 \alpha$ such that $\alpha$ is not a simple root of $R$.
\end{remark}

Let $\alpha \in \Phi_+$ be a root with $k_{\alpha} >1$ and let $R \in \partial^2 \alpha$ be a residue such that $\alpha$ is not a simple root of $R$. Let $P \neq P' \in \partial \alpha$ be the two parallel panels contained in $R$. Then $\ell(1_W, \proj_P 1_W) \neq \ell(1_W, \proj_{P'} 1_W)$ and we can assume that $\ell(1_W, \proj_P 1_W) < \ell(1_W, \proj_{P'} 1_W)$. Let $G = (c_0, \ldots, c_k) \in \mathrm{Min}$ be of type $(s_1, \ldots, s_k)$ such that $c_i = \proj_R 1_W$ for some $0 \leq i \leq k$, $c_i, \ldots, c_k \in R$, $c_{k-1} = \proj_P 1_W$ and $c_k \in P \backslash \{c_{k-1}\}$. For $P \neq Q := \{ x, y \} \in \partial \alpha$ with $x\in \alpha$ and $y \notin \alpha$ we let $P_0 = P, \ldots, P_n = Q$ and $R_1, \ldots, R_n$ be as in Lemma \ref{CM06Prop2.7}.

\begin{lemma}\label{corbasisroot}
	We have $\proj_{R_n} 1_W = \proj_{P_{n-1}} 1_W$, if at least one of the following holds:
	\begin{enumerate}[label=(\alph*)]
		\item $R_1 \neq R$ and $\ell(s_1 \cdots s_{k-1}r) = k$, where $\{r, s_k\}$ is the type of $R_1$.
		
		\item $n>1$.
	\end{enumerate}
\end{lemma}
\begin{proof}
	Suppose $R_1 \neq R$ and $\ell(s_1 \cdots s_{k-1}r) = k$. Then $\proj_{R_1} c_0 = \proj_{P_0} c_0$ and the claim follows from Lemma \ref{projRiinduction}. Now we suppose $n>1$. Assume that $R_1 = R$. Then Lemma \ref{wordsincoxetergroup} implies $\proj_{R_2} 1_W = \proj_{P_1} 1_W$ and the claim follows from Lemma \ref{projRiinduction}. Now we suppose $R_1 \neq R$. If $\ell(s_1 \cdots s_{k-1}r) = k$, the claim follows from assertion $(a)$. Thus we can assume that $\ell(s_1 \cdots s_{k-1}r) = k-2$. Then Lemma \ref{wordsincoxetergroup} yields $k-1 = i+1$, where $c_i = \proj_R c_0$. Define $d := \proj_{R_1} c_0$ and replace $G$ by a minimal gallery $(d_0 = c_0, \ldots, d, c_{k-1}, c_k)$. Now we are in the case $R_1 = R$ and the claim follows.
\end{proof}

\begin{lemma}\label{Lemma:Palpha}
	We have $k = k_{\alpha}$ and the panel $P_{\alpha} := P$ is the unique panel in $\partial \alpha$ with the property $\ell(1_W, \proj_{P_{\alpha}} 1_W) = k_{\alpha} -1$.
\end{lemma}
\begin{proof}
	We have $\ell(1_W, \proj_P 1_W) = k-1$. Thus it suffices to show $\ell(1_W, \proj_Q 1_W) >k-1$. For $n=1$ we obtain $\ell(1_W, \proj_Q 1_W) \geq k$. Now we assume $n>1$. Then we have $\proj_{R_n} 1_W = \proj_{P_{n-1}} 1_W$ by Lemma \ref{corbasisroot}. Since $Q \subseteq R_n$, we deduce $\ell(1_W, \proj_Q 1_W) \geq \ell(1_W, \proj_{R_n} 1_W) = \ell(1_W, \proj_{P_{n-1}} 1_W)$ and the claim follows by induction.
\end{proof}

\begin{lemma}\label{projgal}
	We define $R_{\alpha, Q}$ to be the residue $R_1$ if $R \neq R_1$ and $\ell(s_1 \cdots s_{k-1} r) = k-2$. In all other cases, we define $R_{\alpha, Q} := R$. Then there exists a minimal gallery $H = (d_0 = c_0, \ldots, d_m = \proj_Q c_0, y)$ with the following properties:
	\begin{itemize}
		\item There exists $0 \leq i \leq m$ such that $d_i = \proj_{R_{\alpha, Q}} 1_W$.
		
		\item For each $i+1 \leq j \leq m$ there exists $L_j \in \partial^2 \alpha$ with $\{ c_{j-1}, c_j \} \subseteq L_j$. In particular, we have $d_j \in \mathcal{C}(\partial^2\alpha)$.
	\end{itemize}
\end{lemma}
\begin{proof}
	We define 
	\[ d := \begin{cases}
		\proj_{P_0} c_0 & \text{if } R \neq R_1 \text{ and } \ell(s_1 \cdots s_{k-1}r) = k, \\
		\proj_{P_1} c_0 & \text{else}.
	\end{cases} \]
	We first show that $\ell(c_0, \proj_Q c_0) = \ell(c_0, \proj_{R_{\alpha, Q}} c_0) +  \ell(\proj_{R_{\alpha, Q}} c_0, d) + \ell( d, \proj_Q c_0 )$. By definition we have $R_{\alpha, Q} = R_{\alpha, P_i}$ for all $1 \leq i \leq n$. We prove the hypothesis by induction on $n$. Suppose first $n=1$ and that one of the following hold:
	\begin{itemize}
		\item $R = R_1$;
		
		\item $R\neq R_1$ and $\ell(s_1 \cdots s_{k-1}r) = k-2$;
	\end{itemize}
	Then $Q = P_1 \subseteq R_{\alpha, Q}, d = \proj_Q c_0$ and the claim follows. We prove the case $R \neq R_1$ and $\ell(s_1 \cdots s_{k-1}r) = k$ together with the case $n>1$ simultaneously. Lemma \ref{corbasisroot} provides in both cases $\proj_{R_n} c_0 = \proj_{P_{n-1}} c_0$. If $n>1$, we have $R_{\alpha, Q} = R_{\alpha, P_{n-1}}$; if $n=1$ we have $P_{n-1} = P_0 \subseteq R_{\alpha, Q}$ and $d = \proj_{P_{n-1}} c_0$. This is used in the third equation below. We compute the following:
	\allowdisplaybreaks
	\begin{align*}
		\ell(c_0, \proj_Q c_0) &= \ell(c_0, \proj_{R_n} c_0) + \ell(\proj_{R_n}c_0, \proj_Q c_0) \\
		&= \ell(c_0, \proj_{P_{n-1}} c_0) + \ell(\proj_{P_{n-1}} c_0, \proj_Q c_0) \\
		&= \ell(c_0, \proj_{R_{\alpha, Q}} c_0) + \ell(\proj_{R_{\alpha, Q}} c_0, d) + \ell(d, \proj_{P_{n-1}} c_0) \\
		& \qquad + \ell(\proj_{P_{n-1}} c_0, \proj_Q c_0) \\
		&\geq \ell(c_0, \proj_{R_{\alpha, Q}} c_0) + \ell(\proj_{R_{\alpha, Q}} c_0, d) + \ell(d, \proj_Q c_0) \\
		&\geq \ell(c_0, \proj_Q c_0)
	\end{align*}
	Thus concatenating a minimal gallery from $c_0$ to $\proj_{R_{\alpha, Q}} c_0$, a minimal gallery from $\proj_{R_{\alpha, Q}} c_0$ to $d$ and a minimal gallery from $d$ to $\proj_Q c_0$ yields a minimal gallery from $c_0$ to $\proj_Q c_0$. Using Lemma \ref{CM05Lem2.3} there exists a minimal gallery from $d$ to $\proj_Q c_0$ such that every chamber of this gallery is contained in $\mathcal{C}(\partial^2 \alpha)$ and for two adjacent chambers there exists a residue in $\partial^2 \alpha$ containing both. Since $R_{\alpha, Q} \in \{R, R_1\} \subseteq \partial^2 \alpha$ and, as $R_{\alpha, Q}$ is convex, each chamber of a minimal gallery from $\proj_{R_{\alpha, Q}} c_0$ to $d$ is contained in $R_{\alpha, Q}$ the claim follows.
\end{proof}

\begin{remark}
	In the lemma we use the following notation: For a minimal gallery $G = (c_0, \ldots, c_k)$, $k \geq 1,$ we denote the unique root containing $c_{k-1}$ but not $c_k$ by $\alpha_G$.
\end{remark}

\begin{lemma}\label{prenilpotentrootsintersectinoneresidue}
	Suppose that $(W, S)$ is of type $(4, 4, 4)$ and suppose $S = \{ s_{k-1}, s_k, r \}$. Let $\beta \in \Phi_+ \backslash \{ \alpha_s \mid s\in S \}$ be a root with $k_{\beta} \leq k$, $o(r_{\alpha} r_{\beta}) < \infty$ and $R \notin \partial^2 \beta$. Moreover, we assume that $\ell(s_1 \cdots s_{k-1}r) = k$. Then one of the following hold:
	\begin{enumerate}[label=(\alph*)]
		\item $\beta = \alpha_F$, where $F$ is the minimal gallery of type $(s_1, \ldots, s_{k-1}, r)$;
		
		\item $\beta = \alpha_F$, where $F$ is the minimal gallery of type $(s_1, \ldots, s_{k-2}, s_k, s_{k-1}, r)$, and we have $\ell(s_1 \cdots s_{k-2} s_k r) = k-2$.
	\end{enumerate}
\end{lemma}
\begin{proof}
	Recall that $\alpha = \alpha_G$. As $R \in \partial^2 \alpha$, we have $\alpha \neq \pm \beta$. By Lemma \ref{Lemma: nested infinite order}$(a)$ there exists $C \in \partial^2 \alpha \cap \partial^2 \beta$. Then there exists a panel $Q' \in \partial \alpha$ which is contained in $C$. We let $\proj_{Q'} c_0 \neq y \in Q'$. Let $P_i, R_i$ as before (with $P_n = Q'$), let $G' := (c_0, \ldots, c_{k-1})$ and let $G'' := (c_0, \ldots, c_k, c_{k+1})$ be the minimal gallery of type $(s_1, \ldots, s_k, s_{k-1})$. Let $E$ be a minimal gallery from $c_0$ to $y$ as in Lemma \ref{projgal}. We can extend this minimal gallery (if necessary) to a minimal gallery from $c_0$ to $e \in C$, where $\ell(e) = \ell(\proj_C c_0) +4$. Let $Q'' \in \partial \beta$ be a panel contained in $C$ and let $\proj_{Q''} c_0 \neq y' \in Q''$. Let $H = (d_0 = c_0, \ldots, d_{m-2} = \proj_{R_{\beta, Q''}} c_0, \ldots, d_q := \proj_{Q''} d_0, d_{q+1} := y')$ be a minimal gallery as in Lemma \ref{projgal}. Then $m = k_{\beta} \leq k$. As before, we can extend $H$ (if necessary) to a minimal gallery from $d_0$ to $e$. Note that $R \neq C$ by assumption, and since $R \in \partial^2 \alpha_{G'} \cap \partial^2 \alpha_{G''}, R \notin \partial^2 \beta$ we have $\alpha_{G'} \neq \pm \beta \neq \alpha_{G''}$.
	\begin{enumerate}[label=(\roman*)]
		\item Assume that $R = R_1$: Since $R \in \partial^2 \alpha_{G''} \cap \partial^2 \alpha, C \in \partial^2 \alpha$ and $\alpha_{G''} \neq \pm \beta$, Lemma \ref{Lemma: residues in the 2-boundary are parallel}$(b)$ implies $C \notin \partial^2 \alpha_{G''}$ and hence the gallery $H$ has to cross the wall $\partial \alpha_{G''}$. Assume that $(d_0, \ldots, d_{m-2})$ crosses the wall $\partial \alpha_{G''}$. Let $1 \leq j \leq m-2$ be such that $\{ d_{j-1}, d_j \} \in \partial \alpha_{G''}$. Then $k = k_{\alpha_{G''}} \leq j \leq m-2 \leq k-2$ which is a contradiction. Thus the gallery $(d_0, \ldots, d_{m-2})$ does not cross the wall $\partial \alpha_{G''}$ and hence $(d_{m-1}, \ldots, d_{q+1})$ has to cross the wall $\partial \alpha_{G''}$. Let $m \leq j \leq q+1$ be such that $\{ d_{j-1}, d_j \} \in \partial \alpha_{G''}$. By Lemma \ref{projgal} there exists $L \in \partial^2 \beta$ such that $\{ d_{j-1}, d_j \} \subseteq L$. Then $L \in \partial^2 \beta \cap \partial^2 \alpha_{G''}$ and hence $o(r_{\alpha_{G''}} r_{\beta}) < \infty$. As $\partial^2 \alpha \cap \partial^2 \alpha_{G''} = \{R\} \neq \{C\} = \partial^2 \alpha \cap \partial^2 \beta$ (cf.\ Lemma \ref{Lemma: residues in the 2-boundary are parallel}$(b)$), we have $\partial^2 \alpha \cap \partial^2 \beta \cap \partial^2 \alpha_{G''} = \emptyset$ and hence $\{ r_{\alpha}, r_{\alpha_{G''}}, r_{\beta} \}$ is a reflection triangle. As $\proj_R c_0 \in R \cap \beta \neq \emptyset$ and $R \notin \partial^2 \beta$, we deduce $R \subseteq \beta$. As $e \in C \cap (-\alpha_{G''}) \neq \emptyset$ and $C \notin \partial^2 \alpha_{G''}$, we deduce $C \subseteq (-\alpha_{G''})$. As $L \in \partial^2 \alpha_{G''} \cap \partial^2 \beta, \{ d_{j-1}, d_j \} \subseteq L \cap \alpha$ and $L \notin \partial^2 \alpha$, we deduce $L \subseteq \alpha$. Thus $T := \{ \alpha, -\alpha_{G''}, \beta \}$ is a triangle. For $d \in W$ with $\delta(c_{k-2}, d) = s_k s_{k-1}$ we have $d\in \bigcap_{\gamma \in T} \gamma$ and Lemma \ref{reflectiontrianglechamber}$(b)$ implies $\bigcap_{\gamma \in T} \gamma = \{ d \}$. If $\ell(s_1 \cdots s_{k-2}s_kr) = k$, then $k_{\beta} = k+1$. Thus $\ell(s_1 \cdots s_{k-2}s_kr) = k-2$ and $(b)$ follows.
		
		\item Assume that $R \neq R_1$: Since $R \in \partial^2 \alpha_{G'} \cap \partial^2 \alpha$ and $R \neq C \in \partial^2 \alpha$, Lemma \ref{Lemma: residues in the 2-boundary are parallel}$(b)$ implies $C \notin \partial^2 \alpha_{G'}$ and hence $H$ has to cross the wall $\partial \alpha_{G'}$. Suppose that $(d_0, \ldots, d_{m-2})$ does not cross the wall $\partial^2 \alpha_{G'}$. Replacing $\alpha_{G''}$ by $\alpha_{G'}$ in $(i)$ we obtain that $T := \{ \alpha, -\alpha_{G'}, \beta \}$ is a triangle. Using Lemma \ref{reflectiontrianglechamber}$(b)$, we have $\bigcap_{\gamma \in T} \gamma = \{ c_{k-1} \}$ and hence $(a)$ follows. Now we suppose that $(d_0, \ldots, d_{m-2})$ crosses the wall $\partial \alpha_{G'}$ and let $1 \leq j \leq m-2$ be such that $P' := \{ d_{j-1}, d_j \} \in \partial \alpha_{G'}$. Note that $1 \leq m-2 \leq k-2$ and hence $k \geq 3$. Let $Z$ be the $\{ s_{k-1}, r \}$-residue containing $c_{k-2}$. Then $\alpha_{G'}$ is not a simple root of $Z$ and hence $k_{\alpha_{G'}} \in \{ k-2, k-1 \}$. This implies $k-2 \leq k_{\alpha_{G'}} \leq j \leq m-2 \leq k-2$. Lemma \ref{Lemma:Palpha} implies $P' = P_{\alpha_{G'}}$ and hence $P'$ is contained in $Z$. Moreover, we have $j = m-2$ and $R_{\beta, Q''} = R_{\{r, s_k\}}(d_j)$. Both non-simple roots of $R_{\beta, Q''}$ contain $-\alpha$ by Lemma \ref{mingallinrep}. As one of them is equal to $\beta$, we have a contradiction. \qedhere
	\end{enumerate}
\end{proof}

\section{Commutator blueprints}\label{Section: commutator blueprints}

\begin{convention}
	From now on we assume $m_{st} \in \{2, 3, 4, \infty\}$ for all $s\neq t\in S$.
\end{convention}

We let $\mathcal{P}$ be the set of prenilpotent pairs of positive roots. For $w\in W$ we define $\Phi(w) := \{ \alpha \in \Phi_+ \mid w \notin \alpha \}$. Let $G = (c_0, \ldots, c_k) \in \mathrm{Min}$ and let $(\alpha_1, \ldots, \alpha_k)$ be the sequence of roots crossed by $G$. We define $\Phi(G) := \{ \alpha_i \mid 1 \leq i \leq k \}$. Using the indices we obtain an ordering $\leq_G$ on $\Phi(G)$ and, in particular, on $[\alpha, \beta] = [\beta, \alpha] \subseteq \Phi(G)$ for all $\alpha, \beta \in \Phi(G)$. Note that $\Phi(G) = \Phi(w)$ holds for every $G \in \mathrm{Min}(w)$. We abbreviate $\mathcal{I} := \{ (G, \alpha, \beta) \in \mathrm{Min} \times \Phi_+ \times \Phi_+ \mid \alpha, \beta \in \Phi(G), \alpha \leq_G \beta \}$. 

Given a family $\left(M_{\alpha, \beta}^G \right)_{(G, \alpha, \beta) \in \mathcal{I}}$, where $M_{\alpha, \beta}^G \subseteq (\alpha, \beta)$ is ordered via $\leq_G$. For $w\in W$ we define the group $U_w$ via the following presentation:
\[ U_w := \left\langle \{ u_{\alpha} \mid \alpha \in \Phi(w) \} \;\middle|\; \begin{cases*}
	\forall \alpha \in \Phi(w): u_{\alpha}^2 = 1, \\
	\forall (G, \alpha, \beta) \in \mathcal{I}, G \in \mathrm{Min}(w): [u_{\alpha}, u_{\beta}] = \prod\nolimits_{\gamma \in M_{\alpha, \beta}^G} u_{\gamma}
\end{cases*} \right\rangle
\]
Here the product $\prod\nolimits_{\gamma \in M_{\alpha, \beta}^G} u_{\gamma}$ is understood to be ordered via the ordering $\leq_G$, i.e.\ if $(G, \alpha, \beta) \in \mathcal{I}$ with $G \in \mathrm{Min}(w)$ and $M_{\alpha, \beta}^G = \{ \gamma_1 \leq_G \ldots \leq_G \gamma_k \} \subseteq (\alpha, \beta) \subseteq \Phi(G)$, then $\prod\nolimits_{\gamma \in M_{\alpha, \beta}^G} u_{\gamma} = u_{\gamma_1} \cdots u_{\gamma_k}$. Note that there could be $G, H \in \mathrm{Min}(w), \alpha, \beta \in \Phi(w)$ with $\alpha \leq_G \beta$ and $\beta \leq_H \alpha$. In this case we have two commutation relations, namely 
\begin{align*}
	&[u_{\alpha}, u_{\beta}] = \prod\nolimits_{\gamma \in M_{\alpha, \beta}^G} u_{\gamma} &&\text{and} &&[u_{\beta}, u_{\alpha}] = \prod\nolimits_{\gamma \in M_{\beta, \alpha}^H} u_{\gamma}
\end{align*}
From now on we will implicitly assume that each product $\prod\nolimits_{\gamma \in M_{\alpha, \beta}^G} u_{\gamma}$ is ordered via the ordering $\leq_G$.

\begin{definition}\label{Definition: commutator blueprint}
	A \emph{commutator blueprint of type $(W, S)$} is a family $\mathcal{M} = \left(M_{\alpha, \beta}^G \right)_{(G, \alpha, \beta) \in \mathcal{I}}$ of subsets $M_{\alpha, \beta}^G \subseteq (\alpha, \beta)$ ordered via $\leq_G$ satisfying the following axioms:
	\begin{enumerate}[label=(CB\arabic*)]
		\item Let $G = (c_0, \ldots, c_k) \in \mathrm{Min}$ and let $H = (c_0, \ldots, c_m)$ for some $1 \leq m \leq k$. Then $M_{\alpha, \beta}^H = M_{\alpha, \beta}^G$ holds for all $\alpha, \beta \in \Phi(H)$ with $\alpha \leq_H \beta$.
		
		\item Suppose $s\neq t \in S$ with $m_{st} < \infty$. For $(G, \alpha, \beta) \in \mathcal{I}$ with $G \in \mathrm{Min}(r_{\{s, t\}})$ we have
		\[ M_{\alpha, \beta}^G = \begin{cases}
			(\alpha, \beta) & \{ \alpha, \beta \} = \{ \alpha_s, \alpha_t \} \\
			\emptyset & \{ \alpha, \beta \} \neq \{ \alpha_s, \alpha_t \}
		\end{cases} \]
	
		\item For each $w\in W$ we have $\vert U_w \vert = 2^{\ell(w)}$, where $U_w$ is defined as above.
	\end{enumerate}
	
	A commutator blueprint $\mathcal{M} = \left(M_{\alpha, \beta}^G \right)_{(G, \alpha, \beta) \in \mathcal{I}}$ is called \emph{Weyl-invariant} if for all $w\in W$, $s\in S$, $G \in \mathrm{Min}_s(w)$ and $\alpha, \beta \in \Phi(G) \backslash \{ \alpha_s \}$ with $\alpha \leq_G \beta$ we have $M_{s\alpha, s\beta}^{sG} = sM_{\alpha, \beta}^G := \{ s\gamma \mid \gamma \in M_{\alpha, \beta}^G \}$.
\end{definition}

\begin{lemma}\label{Lemma: Definition of UG}
	Let $\mathcal{M} = \left( M_{\alpha, \beta}^G \right)_{(G, \alpha, \beta) \in \mathcal{I}}$ be a commutator blueprint of type $(W, S)$. Let $w\in W, G = (c_0, \ldots, c_k) \in \mathrm{Min}(w)$ and let $(\alpha_1, \ldots, \alpha_k)$ be the sequence of roots crossed by $G$. Then $\Phi(w) = \{ \alpha_1, \ldots, \alpha_k \}$ and the group $U_w$ has the following presentation:
	\[ U_G := \left\langle u_{\alpha_1}, \ldots, u_{\alpha_k} \;\middle|\; \forall 1 \leq i \leq j \leq k: \quad u_{\alpha_i}^2 = 1, \quad [u_{\alpha_i}, u_{\alpha_j}] = \prod\nolimits_{\gamma \in M_{\alpha_i, \alpha_j}^G} u_{\gamma} \right\rangle \]
\end{lemma}
\begin{proof}
	This is \cite[Lemma $3.6$]{BiRGD}.
\end{proof}

\subsection*{Pre-commutator blueprints}

\begin{definition}
	A \emph{pre-commutator blueprint of type $(W, S)$} is a family $\mathcal{M} = \left(M_{\alpha, \beta}^G \right)_{(G, \alpha, \beta) \in \mathcal{I}}$ of subsets $M_{\alpha, \beta}^G \subseteq (\alpha, \beta)$ ordered via $\leq_G$ satisfying (CB$1$) and (CB$2$) from Definition \ref{Definition: commutator blueprint}, and additionally the following axiom:
	\begin{enumerate}[label=(PCB)]
		\item For all $w\in W$ and $G\in \mathrm{Min}(w)$ the canonical homomorphism $U_G \to U_w$ is an isomorphism, where $U_G$ is defined as in Lemma \ref{Lemma: Definition of UG}.
	\end{enumerate}
\end{definition}

\begin{lemma}\label{Lemma: Conditions to extend U_w}
	Let $\mathcal{M} = \left(M_{\alpha, \beta}^G \right)_{(G, \alpha, \beta) \in \mathcal{I}}$ be a pre-commutator blueprint of type $(W, S)$. Then the following are equivalent:
	\begin{enumerate}[label=(\roman*)]
		\item $\mathcal{M}$ is a commutator blueprint.
		
		\item\label{Lemma: Conditions to extend U_w ii} Let $G = (c_0, \ldots, c_{k+1}) \in \mathrm{Min}$ and let $(\alpha_1, \ldots, \alpha_{k+1})$ be the sequence of roots crossed by $G$. Define $H := (c_0, \ldots, c_k)$ and assume that the following hold in the group $U_H$ for all $1 \leq i \leq j \leq k$:
		\begin{enumerate}[label=(C\arabic*)]
			\item $\left( \prod_{\gamma \in M_{\alpha_i, \alpha_{k+1}}^G} u_{\gamma} \right) \cdot \left( \prod_{\gamma \in M_{\alpha_i, \alpha_{k+1}}^G} \left( u_{\gamma} \prod_{\omega \in M_{\gamma, \alpha_{k+1}}^G} u_{\omega} \right) \right) = 1$;
			
			\item $\left( u_{\alpha_i} \prod_{\gamma \in M_{\alpha_i, \alpha_{k+1}}^G} u_{\gamma} \right)^2 = 1$;
			
			\item $\left[ u_{\alpha_i} \prod_{\gamma \in M_{\alpha_i, \alpha_{k+1}}^G} u_{\gamma}, u_{\alpha_j} \prod_{\gamma \in M_{\alpha_j, \alpha_{k+1}}^G} u_{\gamma} \right] = \prod_{\gamma \in M_{\alpha_i, \alpha_j}^G} \left( u_{\gamma} \prod_{\omega \in M_{\gamma, \alpha_{k+1}}^G} u_{\omega} \right)$;
		\end{enumerate}
	\end{enumerate}
\end{lemma}
\begin{proof}
	One can easily check that the conditions in \ref{Lemma: Conditions to extend U_w ii} must be satisfied, if $\mathcal{M}$ is a commutator blueprint. Thus we can assume \ref{Lemma: Conditions to extend U_w ii}. It suffices to show $\vert U_w \vert = 2^{\ell(w)}$. We will show this by induction. For $\ell(w) \leq 1$ there is nothing to show. Thus we can assume $\ell(w) >1$. Let $G = (c_0, \ldots, c_{k+1}) \in \mathrm{Min}(w)$ and let $(\alpha_1, \ldots, \alpha_{k+1})$ be the sequence of roots crossed by $G$. We define $H := (c_0, \ldots, c_k)$. Using induction, we have $\vert U_H \vert = 2^k$. We will show that $U_G \cong U_H \rtimes_{\phi} \ZZ_2$, where $\phi \in \Aut(U_H)$ acts on $U_H$ as $u_{\alpha_{k+1}}$. We consider the mapping
	\[ u_{\alpha_{k+1}}: \{ u_{\alpha} \mid c_k \notin \alpha  \} \to U_H, u_{\alpha} \mapsto u_{\alpha} \prod\nolimits_{\gamma \in M_{\alpha, \alpha_{k+1}}^G} u_{\gamma} \]
	By (C2) and (C3), $u_{\alpha_{k+1}}$ extends to an endomorphism on $U_H$. Condition (C1) guarantees $u_{\alpha_{k+1}}^2 = \id$ and hence $u_{\alpha_{k+1}} \in \Aut(U_H)$. Thus we can define the semi-direct product $U_H \rtimes_{u_{\alpha_{k+1}}} \ZZ_2$. As $U_w \cong U_G \cong U_H \rtimes_{u_{\alpha_{k+1}}} \ZZ_2$, the claim follows.
\end{proof}

We will see in Theorem \ref{Theorem: pre-com = com} that there are weaker conditions than those in Lemma \ref{Lemma: Conditions to extend U_w}\ref{Lemma: Conditions to extend U_w ii}, which imply that a pre-commutator blueprint is a commutator blueprint. In this case the groups $U_w$ will be nilpotent of class at most $2$. For that we will need the following two preparatory lemmas.

\begin{lemma}\label{nilpotencyclass2}
	Let $G$ be a group and let $X\subseteq G$ be a symmetric generating set (i.e.\ $X = X^{-1}$) such that $[x, [y, z]] = 1$ holds for all $x, y, z \in X$. Then $G$ is of nilpotent of class at most $2$.
\end{lemma}
\begin{proof}
	Suppose $x, y, z \in G$ and let $x_1, \ldots, x_k, y_1, \ldots, y_l, z_1, \ldots, z_m \in X$ such that $x = x_1 \cdots x_k, y = y_1 \cdots y_l, z = z_1 \cdots z_m$. We will show $[x, [y, z]] = 1$. We first assume $l = 1 = m$. Induction on $k$ yields $[x, [y, z]] = [xx_k^{-1}, [y, z]]^{x_k} [x_k, [y, z]] = 1$. Now assume $l=1$. Induction on $m$ implies $[x, [y, z]] = [x, [y, z_m] [y, zz_m^{-1}]^{z_m}] = [x^{(z_m^{-1})}, [y, zz_m^{-1}]]^{z_m} [x, [y, z_m]]^{[y, zz_m^{-1}]^{z_m}} = 1$. Now induction on $l$ yields
	\[ [x, [y, z]] = [x, [yy_l^{-1}, z]^{y_l} [y_l, z]] = [x, [y_l, z]] [x^{(y_l^{-1})}, [yy_l^{-1}, z]]^{y_l[y_l, z]} = 1. \qedhere \]
\end{proof}

\begin{lemma}\label{commutator2}
	Let $G$ be a nilpotent group of class at most $2$ which is generated by a set $X$ of involutions (i.e.\ by elements $g\in G$ with $g^2 = 1$). Then $[g, h]^2 = 1$ holds for all $g, h \in G$.
\end{lemma}
\begin{proof}
	Let $f_1, \ldots, f_r, h_1, \ldots, h_m \in X$ such that $g = f_1 \cdots f_r, h = h_1 \cdots h_m$. We show the claim by induction on $r+m$. If $r+m \in \{0, 1\}$ the claim follows directly. Thus we assume $r+m \geq 2$. Again, if $0 \in \{r, m\}$ the claim follows directly. Thus we can assume $r, m \geq 1$. Using the nilpotency class we obtain
	\allowdisplaybreaks
	\begin{align*}
		[g, h]^2 &= \left( [g, h_m] [g, hh_m^{-1}]^{h_m} \right)^2 \\
		&= \left( [gf_r^{-1}, h_m]^{f_r} [f_r, h_m] [g, hh_m^{-1}] \right)^2 \\
		&= [gf_r^{-1}, h_m]^2 [f_r, h_m]^2 [g, hh_m^{-1}]^2
	\end{align*}
	Using the nilpotency class and the fact that $[f_r, h_m]^2 = [f_r, [f_r, h_m]]$ holds, the claim follows by induction.
\end{proof}

\begin{theorem}\label{Theorem: pre-com = com}
	Let $\mathcal{M} = \left( M_{\alpha, \beta}^G \right)_{(G, \alpha, \beta) \in \mathcal{I}}$ be a pre-commutator blueprint of type $(W, S)$. Then the following are equivalent:
	\begin{enumerate}[label=(\roman*)]
		\item $\mathcal{M}$ is a commutator blueprint of type $(W, S)$ and the groups $U_w$ are nilpotent of class at most $2$.
		
		\item\label{nilpotent conditions} Let $G = (c_0, \ldots, c_{k+1}) \in \mathrm{Min}$ and let $(\alpha_1, \ldots, \alpha_{k+1})$ be the sequence of roots crossed by $G$. For $H := (c_0, \ldots, c_k)$ and for all $1 \leq i < j \leq k +1$ we have the following:
		\begin{enumerate}[label=($2$-n$\arabic*$)]
			\item $\prod_{\gamma \in M_{\alpha_i, \alpha_{k+1}}^G} u_{\gamma} \in Z(U_H)$;
			
			\item $\left( \prod_{\gamma \in M_{\alpha_i, \alpha_{k+1}}^G} u_{\gamma} \right)^2 = 1$ holds in $U_H$;
			
			\item $\prod_{\gamma \in M_{\alpha_i, \alpha_j}^G} \left( u_{\gamma} \prod_{\delta \in M_{\gamma, \alpha_{k+1}}^G} u_{\delta} \right) =\prod_{\gamma \in M_{\alpha_i, \alpha_j}^G} u_{\gamma}$ holds in $U_H$.
		\end{enumerate}
	\end{enumerate}
\end{theorem}
\begin{proof}
	Suppose that $\mathcal{M}$ is a commutator blueprint of type $(W, S)$ and that the groups $U_w$ are nilpotent of class at most $2$. Then $\mathcal{M}$ is a pre-commutator blueprint of type $(W, S)$ by Lemma \ref{Lemma: Definition of UG} and it is not hard to see that $\mathcal{M}$ satisfies ($2$-n$1$), ($2$-n$2$) and ($2$-n$3$) (cf.\ Lemma \ref{commutator2}).
	
	Now we suppose that $\mathcal{M}$ satisfies the conditions in \ref{nilpotent conditions}. We will apply Lemma \ref{Lemma: Conditions to extend U_w}. Let $G = (c_0, \ldots, c_{k+1}) \in \mathrm{Min}(w)$ be a minimal gallery, let $(\alpha_1, \ldots, \alpha_{k+1})$ be the sequence of roots crossed by $G$ and let $H := (c_0, \ldots, c_k)$. Note that (C1) holds by ($2$-n$3$) (applied with $j=k+1$) and ($2$-n$2$). Moreover, (C2) follows from ($2$-n$1$) and ($2$-n$2$). Using ($2$-n$3$) and ($2$-n$1$) one can show that also (C3) holds. Now Lemma \ref{Lemma: Conditions to extend U_w} implies that $\mathcal{M}$ is a commutator blueprint. It is left to show that the group $U_w \cong U_G$ is nilpotent of class at most $2$. We prove this by induction on $\ell(w)$. The claim is obvious for $\ell(w) \leq 1$. Thus we can assume $\ell(w) >1$. Using ($2$-n$3$), we compute for all $1 \leq i < j \leq k+1$:
	\begin{align*}
		\left[ \left[ u_{\alpha_i}, u_{\alpha_j} \right], u_{\alpha_{k+1}} \right] &= \left[ u_{\alpha_i}, u_{\alpha_j} \right]^{-1} \prod\nolimits_{\gamma \in M_{\alpha_i, \alpha_j}^G} u_{\alpha_{k+1}}^{-1} u_{\gamma} u_{\alpha_{k+1}} \\
		&= \left( \prod\nolimits_{\gamma \in M_{\alpha_i, \alpha_j}^G} u_{\gamma} \right)^{-1} \cdot \prod\nolimits_{\gamma \in M_{\alpha_i, \alpha_j}^G} \left( u_{\gamma} \prod\nolimits_{\delta \in M_{\gamma, \alpha_{k+1}}^G} u_{\delta} \right) = 1
	\end{align*}
	Now the claim follows from Lemma \ref{nilpotencyclass2}.
\end{proof}

\section{Examples of commutator blueprints}\label{Section: Examples}

\subsection*{Examples of universal type}

\begin{convention}
	In this subsection we assume that $(W, S)$ is of universal type (i.e.\ $m_{st} = \infty$ for all $s\neq t \in S$) and has rank at least $2$.
\end{convention}

\begin{remark}\label{Remark: PCB and type}
	Note that for $w\in W$ we have $\vert \mathrm{Min}(w) \vert = 1$. This implies that each family $\mathcal{M} = \left( M_{\alpha, \beta}^G \right)_{(G, \alpha, \beta) \in \mathcal{I}}$ of subsets $M_{\alpha, \beta}^G \subseteq (\alpha, \beta)$ ordered via $\leq_G$ satisfies (PCB). Moreover, we note that $(G, \alpha, \beta) \in \mathcal{I}$ implies $\alpha \subseteq \beta$.
\end{remark}

\begin{definition}
	Let $\alpha \in \Phi$ be a root. As $\vert \partial^2 \alpha \vert = 0$, we deduce $\vert \partial \alpha \vert = 1$ (cf.\ Lemma \ref{CM06Prop2.7}) and we call $c^{-1}d \in S$ the \emph{type} of $\alpha$, where $\{ c, d \} \in \partial \alpha$.
\end{definition}

\begin{definition}
	\begin{enumerate}[label=(\alph*)]
		\item Let $s\neq t \in S$, $k \in \NN$, $J \subseteq \{ 1, \ldots, k \}$ and let $(G, \alpha, \beta) \in \mathcal{I}$. Assume that there exists a minimal gallery $H = (c_0, \ldots, c_{2k+1})$ of type $(s, t, \ldots, s, t, s)$ between $\alpha$ and $\beta$ (i.e.\ $\alpha$ is the first and $\beta$ is the last root crossed by $H$) such that $s$ appears $k+1$ times and $t$ appears $k$ times in the type of $H$. Let $(\alpha_1 = \alpha, \ldots, \alpha_{2k+1} = \beta)$ be the sequence of roots crossed by $H$. Then we define
		\[ M(k, J, s, t)_{\alpha, \beta}^G := \{ \alpha_{2j} \mid j\in J \}  \]
		If there does not exist such a gallery between $\alpha$ and $\beta$, we define $M(k, J, s, t)_{\alpha, \beta}^G := \emptyset$.
		
		\item Let $s\neq t \in S$, let $K \subseteq \NN$ be non-empty and let $\mathcal{J} = (J_k)_{k\in K}$ be a family of subsets $J_k \subseteq \{ 1, \ldots, k \}$. For $(G, \alpha, \beta) \in \mathcal{I}$ we define
		\[ M(K, \mathcal{J}, s, t)_{\alpha, \beta}^G := \bigcup\nolimits_{k\in K} M(k, J_k, s, t)_{\alpha, \beta}^G \]
		Moreover, we define $\mathcal{M}(K, \mathcal{J}, s, t) := \left( M(K, \mathcal{J}, s, t)_{\alpha, \beta}^G \right)_{(G, \alpha, \beta) \in \mathcal{I}}$.
	\end{enumerate}
\end{definition}

\begin{remark}
	We note that for given $(G, \alpha, \beta) \in \mathcal{I}$ with $M(K, \mathcal{J}, s, t)_{\alpha, \beta}^G \neq \emptyset$, there exists a unique $k \in K$ such that $M(K, \mathcal{J}, s, t)_{\alpha, \beta}^G = M(k, J_k, s, t)_{\alpha, \beta}^G$.
\end{remark}

\begin{theorem}\label{Theorem: commutator blueprint}
	Let $s\neq t \in S$, let $K \subseteq \NN$ be non-empty and let $\mathcal{J} = (J_k)_{k\in K}$ be a family of subsets $J_k \subseteq \{ 1, \ldots, k \}$. Then $\mathcal{M}(K, \mathcal{J}, s, t)$ is a Weyl-invariant commutator blueprint and the groups $U_w$ are nilpotent of class at most $2$.
\end{theorem}
\begin{proof}
	We abbreviate $\mathcal{M} := \mathcal{M}(K, \mathcal{J}, s, t)$ and $M_{\alpha, \beta}^G := M(K, \mathcal{J}, s, t)_{\alpha, \beta}^G$ for all $(G, \alpha, \beta) \in \mathcal{I}$. By definition, $\mathcal{M}$ satisfies (CB$1$) and (CB$2$), and by Remark \ref{Remark: PCB and type} it also satisfies (PCB). Hence $\mathcal{M}$ is a pre-commutator blueprint. We will apply Theorem \ref{Theorem: pre-com = com}. Thus we suppose $(G, \alpha, \beta) \in \mathcal{I}$ with $M_{\alpha, \beta}^G \neq \emptyset$. Then $\alpha, \beta$ are roots of type $s$ and every $\gamma \in M_{\alpha, \beta}^G$ is a root of type $t$. Now it is straight forward to verify that $\mathcal{M}$ satisfies the conditions in Theorem \ref{Theorem: pre-com = com}\ref{nilpotent conditions}. We infer that $\mathcal{M}$ is a commutator blueprint and the groups $U_w$ are nilpotent of class at most $2$. Moreover, $\mathcal{M}$ is Weyl-invariant, as $M_{\alpha, \beta}^G$ does only depend on the existence of a suitable gallery $H$ and not on $G$.
\end{proof}

\begin{definition}\label{Definition: New rank 2 example}
	Let $s_0, s_1 \in S$ be distinct. Let $(G, \alpha, \beta) \in \mathcal{I}$ and assume that there exists a minimal gallery $H = (c_0, \ldots, c_k)$ of type $(s, t, s, t, \ldots)$ between $\alpha$ and $\beta$. Let $(\alpha_1 = \alpha, \ldots, \alpha_k = \beta)$ be the sequence of roots crossed by $H$. Then we define:
	\begin{align*}
		k=4, s=s_0, t=s_1: &\quad M\left(s_0, s_1\right)_{\alpha_1, \alpha_4}^G = \{ \alpha_2, \alpha_3 \} \\
		k=5, s=s_0, t=s_1: &\quad M\left(s_0, s_1\right)_{\alpha_1, \alpha_5}^G = \{ \alpha_2, \alpha_3 \} \\
		k=5, s=s_1, t=s_0: &\quad M\left(s_0, s_1\right)_{\alpha_1, \alpha_5}^G = \{ \alpha_3, \alpha_4 \} \\
		k=6, s=s_1, t=s_0: &\quad M\left(s_0, s_1\right)_{\alpha_1, \alpha_6}^G = \{ \alpha_3, \alpha_4 \}
	\end{align*}
	Otherwise, we define $M\left(s_0, s_1\right)_{\alpha, \beta}^G := \emptyset$. We define $\mathcal{M}\left( s_0, s_1\right) := \left( M\left( s_0, s_1 \right)_{\alpha, \beta}^G \right)_{(G, \alpha, \beta) \in \mathcal{I}}$.
\end{definition}

\begin{theorem}\label{Theorem: exotic commutator blueprint}
	Let $s_0, s_1 \in S$ be distinct. Then $\mathcal{M}\left( s_0, s_1 \right)$ is a Weyl-invariant commutator blueprint of type $(W, S)$. Moreover, the groups $U_w$ are nilpotent of class at most $2$.
\end{theorem}
\begin{proof}
	We abbreviate $M_{\alpha, \beta}^G := M\left( s_0, s_1 \right)_{\alpha, \beta}^G$ for all $(G, \alpha, \beta) \in \mathcal{I}$. Similarly as in Theorem \ref{Theorem: commutator blueprint}, $\mathcal{M}\left( s_0, s_1 \right)$ is a pre-commutator blueprint and we will apply Theorem \ref{Theorem: pre-com = com}. Let $G = (c_0, \ldots, c_{k+1}) \in \mathrm{Min}$ be a minimal gallery and let $(\alpha_1, \ldots, \alpha_{k+1})$ be the sequence of roots crossed by $G$. We first show the following:
	
	\emph{Claim:} If $\alpha_i$ has type $s_1$ and $\alpha_{i+1}$ has type $s_0$ for some $1 \leq i \leq k$, then $u_{\alpha_i} u_{\alpha_{i+1}} \in Z\left( U_G \right)$.
	
	Let $1 \leq j \leq k+1$. We have to show that $u_j$ commutes with $u_{\alpha_i} u_{\alpha_{i+1}}$. If $j \in \{i, i+1\}$, the claim follows. We next suppose $j<i$. But then it follows by construction, that $M_{\alpha_j, \alpha_i}^G = M_{\alpha_j, \alpha_{i+1}}^G$. This follows for $j>i+1$ similarly.
	
	Now the conditions in Theorem \ref{Theorem: pre-com = com}\ref{nilpotent conditions} follow essentially from the claim. Hence $\mathcal{M}\left( s_0, s_1 \right)$ is a commutator blueprint and the groups $U_w$ are nilpotent of class at most $2$. As $M_{\alpha, \beta}^G$ only depends on the existence of a suitable minimal gallery between $\alpha$ and $\beta$, it is also Weyl-invariant.
\end{proof}

In the rest of this subsection we will introduce a commutator blueprint, where the nilpotency class of the groups $U_w$ becomes arbitrarily large -- and even infinite. In order to do that we will mimic the commutation relations from the upper triangular matrices. We will first define $(i, j, k, l)$-galleries between two roots $\alpha$ and $\beta$ for $i, j, k, l \in \NN$ with $1 \leq i < j$, $1\leq k <l$ and $j<k$. Any $(i, j, k, l)$-gallery will have a certain type $(c_j, c_{j+1} \ldots, c_{k-1}, c_k)$ and the $c_p$ will indicate in which column we are. Roughly speaking, the first root crossed by such a gallery (which is $\alpha$) corresponds to the $(i, j)$-entry of a matrix and the last root crossed by such a gallery (which is $\beta$) corresponds to the $(k+1, l)$-entry or -- if $k=l-1$ -- to the $(1, k+1)$-entry. Every $(i, j, k, l)$-gallery will cross roots corresponding to the the following entries of a matrix:
\[ (i, j), (i+1, j), \ldots, (j-1, j), (1, j+1), \ldots, (k, l) \]
As we will mimic the commutation relations from the upper triangular matrices, the sets $M_{\mathrm{nil}}(n)_{\alpha, \beta}^G$ will correspond to the $(i, k)$-entry, if $j=k$.

\begin{definition}\label{Definition: M_nil n}
	Let $(W, S)$ be of rank at least $3$ and let $r, s, t \in S$ be pairwise distinct. 
	\begin{enumerate}[label=(\alph*)]
		\item For $i \in \NN$ we let $k_i := tsts\cdots$ with $\ell(k_i) = i$, e.g.\ $k_3 = tst$.
		
		\item Suppose $i_1, j_1, i_2, j_2 \in \NN$ with $1 \leq i_1 < j_1$, $1 \leq i_2 < j_2$ and $j_1 < j_2$. A gallery is called \emph{$(i_1, j_1, i_2, j_2)$-gallery} if it has type
		\[ \left( c_{j_1}, \ldots, c_{j_2} \right), \]
		where $c_i$ is defined as follows:
		\allowdisplaybreaks
		\begin{align*}
			r_{i,j} &= \left( r, k_i, r, k_j, r \right), \\
			c_{j_1} &= \left( s, r_{i_1,j_1}, s, r_{i_1+1, j_1}, \cdots, s, r_{j_1-1, j_1} \right) \\
			c_i &= \left( s, r_{1, i}, s, r_{2, i}, \cdots, s, r_{i-1, i} \right) \quad \text{for } j_1 < i < j_2 \\
			c_{j_2} &= \left( s, r_{1, j_2}, s, r_{2, j_2}, \cdots, s, r_{i_2-1, j_2}, s \right)
		\end{align*}
		
		\item Let $n \in \NN \cup \{ \infty \}$ and we define $z\leq \infty$ for all $z\in \NN$. Let $(G, \alpha, \beta) \in \mathcal{I}$ and assume that there exists a $(i_1, j_1, i_2, j_2)$-gallery between $\alpha$ and $\beta$. We let $(\alpha_{i_1, j_1}, \alpha_{i_1+1, j_1}, \ldots, \alpha_{j_1-1, j_1}, \alpha_{1, j_1+1}, \ldots, \alpha_{i_2, j_2})$ be the sequence of roots crossed by this gallery omitting the $r_{i, j}$ parts from above. Then $\alpha = \alpha_{i_1, j_1}$ and $\beta = \alpha_{i_2, j_2}$. We define
		\[ M_{\mathrm{nil}}(n)_{\alpha, \beta}^G := \begin{cases}
			\{ \alpha_{i_1, j_2} \} & \text{if } j_1 = i_2 \text{ and } j_2 \leq n \\
			\emptyset & \text{else}
		\end{cases} \]
		If there is no $(i_1, j_1, i_2, j_2)$-gallery between $\alpha$ and $\beta$ for all $i_1, j_1, i_2, j_2 \in \NN$, then we define $M_{\mathrm{nil}}(n)_{\alpha, \beta}^G := \emptyset$. Moreover, we define $\mathcal{M}_{\mathrm{nil}}(n) := \left( M_{\mathrm{nil}}(n)_{\alpha, \beta}^G \right)_{(G, \alpha, \beta) \in \mathcal{I}}$.
	\end{enumerate}
\end{definition}

\begin{example}
	A $(2, 3, 3, 5)$-gallery $G$ has type $(c_3, c_4, c_5)$, where
	\allowdisplaybreaks
	\begin{align*}
		c_3 &= (s, r_{2, 3}), && r_{2, 3} = (r, t, s, r, t, s, t, r) \\
		c_4 &= (s, r_{1, 4}, s, r_{2, 4}, s, r_{3, 4}), && r_{1, 4} = (r, t, r, t, s, t, s, r) \\
		&&& r_{2, 4} = (r, t, s, r, t, s, t, s, r) \\
		&&& r_{3, 4} = (r, t, s, t, r, t, s, t, s, r) \\
		c_5 &= (s, r_{1, 5}, s, r_{2, 5}, s), && r_{1, 5} = (r, t, r, t, s, t, s, t, r) \\
		&&& r_{2, 5} = (r, t, s, r, t, s, t, s, t, r)
	\end{align*}
	If $(\alpha_1, \ldots, \alpha_{61})$ is the sequence of roots crossed by $G$, then we have the following:
	\allowdisplaybreaks
	\begin{align*}
		&\alpha_{2, 3} = \alpha_1, && \alpha_{1, 4} = \alpha_{10}, && \alpha_{2, 4} = \alpha_{19}, && \alpha_{3, 4} = \alpha_{29}, && \alpha_{1, 5} = \alpha_{40} && \alpha_{2, 5} = \alpha_{50}, && \alpha_{3, 5} = \alpha_{61}.
	\end{align*}
	If $\alpha_1, \alpha_{61} \in \Phi_+$, then we have $M_{\mathrm{nil}}(n)_{\alpha_1, \alpha_{61}}^G = M_{\mathrm{nil}}(n)_{\alpha_{2, 3}, \alpha_{3, 5}}^G = \begin{cases}
		\{ \alpha_{2, 5} = \alpha_{50} \} & \text{if } 5\leq n \\
		\emptyset & \text{else}.
	\end{cases}$
\end{example}

\begin{remark}\label{Remark: M_nil(n)}
	\begin{enumerate}[label=(\alph*)]
		\item Let $\alpha, \beta, \gamma \in \Phi$ be three roots. Assume that $G = (d_0, \ldots, d_k)$ is an $(i_1, j_1, i_2, j_2)$-gallery between $\alpha$ and $\beta$ and that $(e_0, \ldots, e_l)$ is an $(i_2, j_2, i_3, j_3)$-gallery between $\beta$ and $\gamma$. Then $(d_0, \ldots, d_{k-1} = e_0, \ldots, e_l)$ is an $(i_1, j_1, i_3, j_3)$-gallery between $\alpha$ and $\gamma$.
		
		\item Let $\alpha, \beta \in \Phi$ be two roots with $M_{\mathrm{nil}}(n)_{\alpha, \beta}^G = \{ \gamma \}$. Then there exists an $(i, j, j, k)$-gallery between $\alpha$ and $\beta$ for some $i, j, k \in \NN$ with $i<j<k \leq n$. Note that by definition, we have $M_{\mathrm{nil}}(n)_{\alpha, \gamma}^G = \emptyset = M_{\mathrm{nil}}(n)_{\gamma, \beta}^G$.
	\end{enumerate}
\end{remark}

\begin{lemma}\label{Lemma: M_nil(n)}
	Let $n \in \NN \cup \{ \infty \}$. Let $(G, \alpha, \beta) \in \mathcal{I}$ and suppose that there does not exist an $(i, j, i', j')$-gallery between $\alpha$ and $\beta$ for all $i, j, i', j' \in \NN$. Let $\delta \in \Phi(G)$ with $\alpha \leq_G \delta \leq_G \beta$. Then the following hold:
	\begin{enumerate}[label=(\alph*)]
		\item If $\gamma \in M_{\mathrm{nil}}(n)_{\alpha, \delta}^G$, then $M_{\mathrm{nil}}(n)_{\gamma, \beta}^G = \emptyset$.
		
		\item If $\gamma \in M_{\mathrm{nil}}(n)_{\delta, \beta}^G$, then $M_{\mathrm{nil}}(n)_{\alpha, \gamma}^G = \emptyset$.
	\end{enumerate}
\end{lemma}
\begin{proof}
	We abbreviate $M_{\sigma, \rho}^H := M_{\mathrm{nil}}(n)_{\sigma, \rho}^H$ for all $(H, \sigma, \rho) \in \mathcal{I}$. We first show $(a)$. We can assume that $M_{\alpha, \delta}^G = \{ \gamma \}$. By definition, there exists an $(i, j, j, k)$-gallery between $\alpha$ and $\delta$ for some $i, j, k \in \NN$ with $i < j < k \leq n$, and this gallery has type $(c_j, c_{j+1}, \ldots, c_{k-1}, c_k)$. Note that
	\[ c_k = (s, r_{1, k}, s, r_{2, k}, \ldots, s, r_{i-1, k}, \textbf{s}, r_{i, k}, \ldots, s, r_{j-1, k}, s) \]
	and $\gamma$ is the root of type $\textbf{s}$ crossed by $c_k$ directly before $r_{i, k}$.
	
	Assume that $M_{\gamma, \beta}^G \neq \emptyset$. Then there would exist an $(i', j', j', k')$-gallery between $\gamma$ and $\beta$ for some $i', j', k' \in \NN$ with $i' < j' < k' \leq n$. Because the type of $G$ after crossing $\gamma$ is $r_{i, k}$ we know that $i' = i$ and $j' = k$. But then we have an $(i, j, j', k')$-gallery between $\alpha$ and $\beta$ by Remark \ref{Remark: M_nil(n)}, which is a contradiction. This finishes the claim.
	
	We now show $(b)$. Again we can assume $M_{\delta, \beta}^G = \{ \gamma \}$. By definition, there exists an $(i, j, j, k)$-gallery between $\delta$ and $\beta$ for some $i, j, k \in \NN$ with $i<j<k \leq n$, and this gallery has type $(c_j, \ldots, c_k)$. Note that
	\allowdisplaybreaks
	\begin{align*}
		&c_{k-1} = (\ldots, s, r_{k-2, k-1}) &&\text{and} && c_k = (s, r_{1, k}, \ldots, s, r_{j-1, k}, s).
	\end{align*}
	Assume that $M_{\alpha, \gamma}^G \neq \emptyset$. Then there would exist a $(i', j', j', k')$-gallery between $\alpha$ and $\gamma$ for some $i', j', k' \in \NN$ with $i' < j' < k' \leq n$. We distinguish the following cases:	
	\begin{enumerate}[label=(\roman*)]
		\item $i=1$: Then $\gamma$ is the first root crossed by $c_k$ and the type of $G$ just before crossing $\gamma$ is $r_{k-2, k-1}$. This implies that $j' = 1$ and $k' = k$, which is a contradiction.
		
		\item $i>1$: In this case the type of $G$ just before crossing $\gamma$ is $r_{i-1, k}$. This implies $j' = i$ and $k' = k$. But then we have an $(i', j', j, k)$-gallery between $\alpha$ and $\beta$, which is a contradiction. \qedhere
	\end{enumerate}
\end{proof}

\begin{theorem}\label{Theorem: M_nil(n)}
	For each $n \in \NN \cup \{ \infty \}$, $\mathcal{M}_{\mathrm{nil}}(n)$ is a Weyl-invariant commutator blueprint. Moreover, the following hold:
	\begin{enumerate}[label=(\roman*)]
		\item\label{Nilpotency class M_nil(n) bounded} If $n \in \NN_{\geq 3}$, then every group $U_w$ is nilpotent of class at most $n-1$, but not all are nilpotent of class at most $n-2$.
		
		\item\label{Nilpotency class M_nil(n) unbounded} If $n = \infty$, then for every $m \in \NN$ there exists $w_m \in W$ such that $U_{w_m}$ is not nilpotent of class at most $m$.
	\end{enumerate}
\end{theorem}
\begin{proof}
	We abbreviate $\mathcal{M} := \mathcal{M}_{\mathrm{nil}}(n)$ and $M_{\alpha, \beta}^G := M_{\mathrm{nil}}(n)_{\alpha, \beta}^G$ for all $(G, \alpha, \beta) \in \mathcal{I}$. By definition, $\mathcal{M}$ satisfies (CB$1$) and (CB$2$), and by Remark \ref{Remark: PCB and type} it also satisfies (PCB). Hence $\mathcal{M}$ is a pre-commutator blueprint. We will apply Lemma \ref{Lemma: Conditions to extend U_w}. Suppose $G = (d_1, \ldots, d_p) \in \mathrm{Min}$, let $\beta \in \Phi(G)$ be the last root which is crossed by $G$ and let $\alpha \in \Phi(G) \backslash \{ \beta \}$. If $M_{\alpha, \beta}^G = \emptyset$, then (C1) and (C2) are satisfied. Otherwise, there exists an $(i, j, j, k)$-gallery between $\alpha$ and $\beta$ for some $i, j, k \in \NN$ with $1 \leq i < j < k \leq n$ and we have $M_{\alpha, \beta}^G = \{ \gamma \}$. Then $M_{\alpha, \gamma}^G = \emptyset = M_{\gamma, \beta}^G$ holds by definition (cf.\ Remark \ref{Remark: M_nil(n)}). Thus (C1) and (C2) are satisfied and it remains to check (C3).
	
	Let $\alpha, \delta \in \Phi(G) \backslash \{ \beta \}$ with $\alpha \leq_G \beta$. If there is an $(i_1, j_1, i_2, j_2)$-gallery between $\alpha$ and $\beta$ for some $i_1, j_1, i_2, j_2 \in \NN$, then (C3) follows essentially from the fact that (C3) holds for upper triangular matrices. Thus we can assume that there does not exist any $(i_1, j_1, i_2, j_2)$-gallery between $\alpha$ and $\beta$ for all $i_1, j_1, i_2, j_2 \in \NN$. In particular, $M_{\alpha, \beta}^G = \emptyset$ and we have to check that the following is a relation in $U_{(d_0, \ldots, d_{p-1})}$:
	\[ \left[ u_{\alpha}, u_{\delta} \prod\nolimits_{\gamma \in M_{\delta, \beta}^G} u_{\gamma} \right] = \prod\nolimits_{\gamma \in M_{\alpha, \delta}^G} \left( u_{\gamma} \prod\nolimits_{\omega \in M_{\gamma, \beta}^G} u_{\omega} \right) \]
	
	Suppose $M_{\alpha, \delta}^G \neq \emptyset$ and let $\gamma \in M_{\alpha, \delta}^G$. Then Lemma \ref{Lemma: M_nil(n)} implies $M_{\gamma, \beta}^G = \emptyset$ and the right hand side of the previous equation is equal to $\prod\nolimits_{\gamma \in M_{\alpha, \delta}^G} u_{\gamma}$ in both cases $M_{\alpha, \delta}^G = \emptyset$ and $M_{\alpha, \delta}^G \neq \emptyset$.
	
	If $M_{\delta, \beta}^G = \emptyset$, then (C3) holds and we are done. Thus we can assume $M_{\delta, \beta}^G = \{ \gamma \}$. Using Remark \ref{Remark: M_nil(n)}, we deduce $M_{\delta, \gamma}^G = \emptyset$. Moreover, Lemma \ref{Lemma: M_nil(n)} yields $M_{\alpha, \gamma}^G = \emptyset$. This implies that the following is a relation in $U_{(d_0, \ldots, d_{p-1})}$ and hence (C3) holds:
	\[ \left[ u_{\alpha}, u_{\delta} \prod\nolimits_{\omega \in M_{\delta, \beta}^G} u_{\omega} \right] = [ u_{\alpha}, u_{\delta} u_{\gamma} ] = [ u_{\alpha}, u_{\gamma} ] [ u_{\alpha}, u_{\delta} ]^{u_{\gamma}} = [u_{\alpha}, u_{\delta}] =  \prod\nolimits_{\omega \in M_{\alpha, \delta}^G} u_{\omega} \]
	
	Now Lemma \ref{Lemma: Conditions to extend U_w} yields that $\mathcal{M}$ is a commutator blueprint. The Weyl-invariance follows from the fact that $M_{\alpha, \beta}^G$ does not depend on $G$, but only on the existence of a suitable minimal gallery between $\alpha$ and $\beta$.
	
	We will show \ref{Nilpotency class M_nil(n) bounded} and \ref{Nilpotency class M_nil(n) unbounded}. Let $n \in \NN_{\geq 3}$ and let $G \in \mathrm{Min}(w)$ be a $(1, 2, n-1, n)$-gallery. Then $U_w \cong U_G$ is not nilpotent of class at most $n-2$. This proves \ref{Nilpotency class M_nil(n) unbounded} and the second part of \ref{Nilpotency class M_nil(n) bounded}. Thus it is left to show that for every $w\in W$ the group $U_w$ is nilpotent of class at most $n-1$. 
	
	Let $w\in W$. We prove the claim by induction on $\ell(w)$. If $\ell(w) \leq 1$, the claim follows. Thus we can assume $\ell(w) >1$. Let $G = (c_0, \ldots, c_p) \in \mathrm{Min}(w)$, let $(\alpha_1, \ldots, \alpha_p)$ be the sequence of roots crossed by $G$ and let $H := (c_0, \ldots, c_{p-1})$. Using induction, the group $U_H$ is nilpotent of class at most $n-1$. If $M_{\alpha_i, \alpha_p}^G = \emptyset$ for all $1 \leq i \leq p-1$, then $U_w \cong U_H \times \ZZ_2$ and $U_w$ is again nilpotent of class at most $n-1$. Thus we can assume $M_{\alpha_i, \alpha_p}^G \neq \emptyset$ for some $1 \leq i \leq p-1$ and hence there exists a $(j, k, k, l)$-gallery between $\alpha_i$ and $\alpha_p$ for some $j, k, l \in \NN$ with $1 \leq j<k<l\leq n$. Then the type of $G$ ends with 
	\[ (\ldots, s, r_{1, l}, \ldots, s, r_{k-1, l}, s) \]
	Let $1 \leq i' \leq i$ be minimal with the property that there exists an $(i_1, j_1, i_2, j_2)$-gallery between $\alpha_{i'}$ and $\alpha_p$ for some $i_1, j_1, i_2, j_2 \in \NN$ with $1\leq i_1 < j_1, 1 \leq i_2 < j_2$ and $j_1 < j_2$. Note that $M_{\alpha_{i'}, \alpha_p}^G = \emptyset$ is possible. Comparing the type of $G$ and the type of the $(i_1, j_1, i_2, j_2)$-gallery, we deduce $i_2 = k$ and $j_2 = l \leq n$.
	
	Let $\alpha \in \{ \alpha_1, \ldots, \alpha_{i' -1} \}$ and $\beta \in \{ \alpha_{i' +1}, \ldots, \alpha_p \}$ and assume $M_{\alpha, \beta}^G \neq \emptyset$. Then there would exist a $(j', k', k', l')$-gallery between $\alpha$ and $\beta$ for some $j', k', l' \in \NN$ with $j' < k' < l' \leq n$. By definition, $\beta$ must be of type $s$ and the root crossed directly before $\beta$ by $G$ must be of type $r$. Note that $\beta$ is a root on the $(i_1, j_1, i_2, j_2)$-gallery between $\alpha_{i'}$ and $\alpha_p$. By considering the type of this gallery it follows that $\beta$ is a root of the form $\alpha_{x, y}$ (cf.\ Definition \ref{Definition: M_nil n}). But then we would also have an $(i_1', j_1', i_2', j_2')$-gallery between $\alpha$ and $\alpha_p$ by Remark \ref{Remark: M_nil(n)}, which is a contradiction to the minimality of $i'$. Thus we deduce $M_{\alpha, \beta}^G = \emptyset$. However, we note that it is possible that $M_{\alpha, \alpha_{i'}}^G \neq \emptyset \neq M_{\alpha_{i'}, \beta}^G$.
	
	If $i' = 1$, then the claim follows. Thus we can assume $1 < i'$. Using induction, the group $U_{(c_0, \ldots, c_{i'})}$ is nilpotent of class at most $n-1$ and the subgroup of $U_G$ generated by $\{ u_{\alpha_{i'}}, \ldots, u_{\alpha_p} \}$ is also nilpotent of class at most $n-1$. Moreover, we have $u_{\alpha_{i'}} \notin [U_G, U_G]$ because of the type of $G$ and the fact that $i'$ is minimal (cf.\ also Remark \ref{Remark: M_nil(n)}). This implies
	\[ [U_G, U_G] \leq \langle u_{\alpha_i} \mid 1 \leq i \leq i'-1 \rangle \times \langle u_{\alpha_i} \mid i' +1 \leq i \leq p \rangle \]
	It follows that the nilpotency class of $U_G$ is equal to the maximum of the nilpotency classes of the groups $U_{(c_0, \ldots, c_{i'})}$ and the subgroup of $U_w$ generated by $\{ u_{\alpha_{i'}}, \ldots, \alpha_p \}$, which is at most $n-1$. This finishes the claim.
\end{proof}

\subsection*{Examples of type $(4, 4, 4)$}

\begin{convention}
	For the rest of this paper we assume that $(W, S)$ is of rank $3$ and that $m_{st} = 4$ for all $s\neq t \in S$.
\end{convention}

\begin{proposition}\label{Proposition: 2-nilpotent}
	Let $\mathcal{M} = \left( M_{\alpha, \beta}^G \right)_{(G, \alpha, \beta) \in \mathcal{I}}$ be a pre-commutator blueprint of type $(4, 4, 4)$ such that for all $(G, \alpha, \beta) \in \mathcal{I}$ with $o(r_{\alpha} r_{\beta}) < \infty$ we have
	\[ M_{\alpha, \beta}^G = \begin{cases}
		(\alpha, \beta) & \vert (\alpha, \beta) \vert = 2 \\
		\emptyset & \vert (\alpha, \beta) \vert < 2
	\end{cases} \]
	Let $G = (c_0, \ldots, c_k) \in \mathrm{Min}$ and let $(\alpha_1, \ldots, \alpha_k)$ be the sequence of roots crossed by $G$. We let $\beta := \alpha_k$ and assume that the following hold for all $\alpha \in \Phi(G) \backslash \{ \beta \}$:
	\begin{enumerate}[label=(\alph*)]
		\item Suppose that $o(r_{\alpha} r_{\beta}) < \infty$, $M_{\alpha, \beta}^G = \{ \gamma, \delta \}$ and let $\epsilon \in \Phi(G)$. 
		\begin{enumerate}[label=(\roman*)]
			\item If $\epsilon \subsetneq \gamma$ and $\epsilon \subsetneq \delta$, then $M_{\epsilon, \gamma}^G = M_{\epsilon, \delta}^G$ holds.
			
			\item If $\gamma \subsetneq \epsilon$ and $\delta \subsetneq \epsilon$, then $M_{\gamma, \epsilon}^G = M_{\delta, \epsilon}^G$ holds.
		\end{enumerate}
		
		\item Suppose that $o(r_{\alpha} r_{\beta}) = \infty$. Then the following hold:
		\begin{enumerate}[label=(\roman*)]
			\item $\prod\nolimits_{\gamma \in M_{\alpha, \beta}^G} u_{\gamma} \in Z(U_{(c_0, \ldots, c_{k-1})})$.
			
			\item $\left( \prod\nolimits_{\gamma \in M_{\alpha, \beta}^G} u_{\gamma} \right)^2 = 1$ holds in $U_{(c_0, \ldots, c_{k-1})}$.
		\end{enumerate}
	\end{enumerate}
	If ($2$-n$3$) of Theorem \ref{Theorem: pre-com = com}\ref{nilpotent conditions} holds, then $\mathcal{M}$ is a commutator blueprint and the groups $U_w$ are nilpotent of class at most $2$.
\end{proposition}
\begin{proof}
	Note that by assumption and Theorem \ref{Theorem: pre-com = com} it suffices to show that $\mathcal{M}$ satisfies ($2$-n$1$) and ($2$-n$2$). We abbreviate $u_i := u_{\alpha_i}$ for all $1 \leq i \leq k$. 
	
	Let $1 \leq i \leq k-1$. We first see that ($2$-n$2$) holds in the case $o(r_{\alpha_i} r_{\beta}) < \infty$ by assumption and in the case $o(r_{\alpha_i} r_{\beta}) = \infty$ by condition $(b)(ii)$. Thus it suffices to show that ($2$-n$1$) holds. 
	
	Let $1 \leq i \leq k-1$. We have to show $[u_i, [u_{\alpha}, u_{\beta}]] = 1$. If $o(r_{\alpha} r_{\beta}) = \infty$, then $[u_{\alpha}, u_{\beta}]$ commutes with $u_i$ by condition $(b)(i)$ and the claim follows. Thus we assume $o(r_{\alpha} r_{\beta}) < \infty$. Moreover, we can assume that $M_{\alpha, \beta}^G \neq \emptyset$ and hence $\vert (\alpha, \beta) \vert = 2$. We suppose $M_{\alpha, \beta}^G = \{\gamma, \delta\}$ with $\gamma \leq_G \delta$.
		
	If $R \in \partial^2 \alpha_i \cap \partial^2 \alpha \cap \partial^2 \beta$, then the claim follows from the assumptions. Thus we can assume that $\partial^2 \alpha_i \cap \partial^2 \alpha \cap \partial^2 \beta = \emptyset$. Suppose that $o(r_{\alpha_i} r_{\gamma}) = \infty = o(r_{\alpha_i} r_{\delta})$. As $\{ \alpha_i, \gamma \}, \{\alpha_i, \delta\} \in \mathcal{P}$, we deduce from Lemma \ref{Lemma: nested infinite order}$(a)$ that both are nested. As $o(r_{\gamma} r_{\delta}) < \infty$, we deduce $\alpha_i \subseteq \gamma, \delta$ or $\gamma, \delta \subseteq \alpha_i$. If $\alpha_i \subseteq \gamma, \delta$, we have $M_{\alpha_i, \gamma}^G = M_{\alpha_i, \delta}^G$ by condition $(a)(i)$ and we infer
		\allowdisplaybreaks
		\begin{align*}
			u_i^{-1} \left( \prod\nolimits_{\epsilon \in M_{\alpha, \beta}^G} u_{\epsilon} \right) u_i &= \prod\nolimits_{\epsilon \in M_{\alpha, \beta}^G} \left( \prod\nolimits_{\omega \in M_{\alpha_i, \epsilon}^G} u_{\omega} \right) u_{\epsilon} \\
			&= \left( \prod\nolimits_{\omega \in M_{\alpha_i, \gamma}^G} u_{\omega} \right) u_{\gamma} \left( \prod\nolimits_{\omega \in M_{\alpha_i, \delta}^G} u_{\omega} \right) u_{\delta} \\
			&\overset{(b)(i)}{=} u_{\gamma} \left( \prod\nolimits_{\omega \in M_{\alpha_i, \gamma}^G} u_{\omega} \right)^2 u_{\delta} \\
			&\overset{(b)(ii)}{=} u_{\gamma} u_{\delta} = \prod\nolimits_{\epsilon \in M_{\alpha, \beta}^G} u_{\epsilon}
		\end{align*}
		The case $\gamma, \delta \subseteq \alpha_i$ follows similarly. Now we suppose that there exists $\rho \in (\alpha, \beta)$ with $o(r_{\alpha_i} r_{\rho}) < \infty$. Assume that $o(r_{\alpha} r_{\alpha_i}) = \infty = o(r_{\beta} r_{\alpha_i})$. Then we would have $\alpha_i \subseteq \beta$ (as $\beta = \alpha_k$) and, moreover, either $\alpha \subseteq \alpha_i$ or $\alpha_i \subseteq \alpha$. As $o(r_{\alpha} r_{\beta}) < \infty$, we deduce $\alpha_i \subseteq \alpha$, but then we would have $\alpha_i \subseteq \alpha \cap \beta \subseteq \rho$, which is a contradiction. We conclude $o(r_{\alpha} r_{\alpha_i}) < \infty$ or $o(r_{\beta} r_{\alpha_i}) < \infty$. Let $\epsilon \in \{ \alpha, \beta \}$ be a root such that $o(r_{\epsilon} r_{\alpha_i}) < \infty$. Then Lemma \ref{Lemma: nested infinite order}$(b)$ yields
		\begin{align*}
			&o(r_{\epsilon} r_{\rho}) < \infty &&\text{and} &&\partial^2 \alpha \cap \partial^2 \beta \cap \partial^2 \rho \neq \emptyset.
		\end{align*}
		Let $\epsilon' \in \{ \alpha, \beta \} \backslash \{ \epsilon \}$. We deduce from Lemma \ref{Lemma: residue cut by intersection reflections} that $\partial^2 \rho \cap \partial^2 \epsilon = \partial^2 \alpha \cap \partial^2 \beta = \partial^2 \rho \cap \partial^2 \epsilon'$. As $\partial^2 \alpha_i \cap \partial^2 \alpha \cap \partial^2 \beta = \emptyset$, we infer that $\{ r_{\alpha_i}, r_{\epsilon}, r_{\rho} \}$ is a reflection triangle. Using Remark \ref{Remark: reflection triangle plus orientation yields triangle} there exist $\beta_i \in \{ \alpha_i, -\alpha_i \}, \beta_{\epsilon} \in \{ \epsilon, -\epsilon \}$ and $\beta_{\rho} \in \{ \rho, -\rho \}$ such that $\{ \beta_i, \beta_{\epsilon}, \beta_{\rho} \}$ is a triangle. Note that $(-\beta_{\epsilon}, \beta_{\rho}) = \emptyset$ by Lemma \ref{reflectiontrianglechamber}. Let $\sigma \in (\alpha, \beta) \backslash \{ \rho \}$. Using Lemma \ref{Lemma: residue cut by intersection reflections}, there exists $\beta_{\epsilon'} \in \{ \epsilon', -\epsilon' \}$ and $\beta_{\sigma} \in \{ \sigma, -\sigma \}$ with $\beta_{\epsilon'}, \beta_{\sigma} \in (\beta_{\epsilon}, \beta_{\rho})$. By Lemma \ref{Lemma: Triangle and infinite order} we have $o(r_{\alpha_i} r_{\epsilon'}) = \infty$. Note that $\{ \alpha_i, \epsilon' \} \in \P$ and hence $\{ \alpha_i, \epsilon' \}$ is nested by Lemma \ref{Lemma: nested infinite order}. Recall that $\partial^2 \epsilon \cap \partial^2 \epsilon' \cap \partial^2 \rho = \partial^2 \alpha \cap \partial^2 \beta \cap \partial^2 \rho \neq \emptyset$. We distinguish the following two cases:
		\begin{enumerate}[label=(\alph*)]
			\item $\alpha_i \subseteq \epsilon'$: For $R \in \partial^2 \epsilon \cap \partial^2 \rho = \partial^2 \epsilon' \cap \partial^2 \rho = \partial^2 \alpha \cap \partial^2 \beta$ (cf.\ Lemma \ref{Lemma: residue cut by intersection reflections}), we deduce $\emptyset \neq R\cap (-\epsilon') \subseteq (-\alpha_i)$ and, as $R \notin \partial^2 \alpha_i$, we have $R \subseteq (-\alpha_i)$. This yields $\beta_i = -\alpha_i$. For $R \in \partial^2 \alpha_i \cap \partial^2 \epsilon$ we have $\emptyset \neq \alpha_i \cap R \subseteq \epsilon'$. As $\partial^2 \alpha_i \cap \partial^2 \epsilon \cap \partial^2 \epsilon' = \partial^2 \alpha_i \cap \partial^2 \alpha \cap \partial^2 \beta = \emptyset$, we deduce $R \notin \partial^2 \epsilon'$ and hence $R \subseteq \epsilon'$. In particular, we have $\emptyset \neq \epsilon \cap R \subseteq \epsilon \cap \epsilon' = \alpha \cap \beta \subseteq \rho$. As $R \notin \partial^2 \rho$ ($\{ r_{\alpha_i}, r_{\epsilon}, r_{\rho} \}$ is a reflection triangle), we infer $R \subseteq \rho$ and hence $\beta_{\rho} = \rho$. Lemma \ref{reflectiontrianglechamber} implies $(\alpha_i, \rho) = (-\beta_i, \beta_{\rho}) = \emptyset$. Recall that $\beta_{\sigma} \in (\beta_{\epsilon}, \beta_{\rho})$. Using Lemma \ref{Lemma: Triangle and infinite order}, we deduce $o(r_{\alpha_i} r_{\sigma}) = \infty, \alpha_i = -\beta_i \subseteq \beta_{\sigma}$ and $(\alpha_i, \beta_{\sigma}) = \emptyset$. As $\alpha_i \ni 1_W \notin (-\sigma)$, we deduce $\beta_{\sigma} = \sigma$. Since $\{ \gamma, \delta \} = (\alpha, \beta) = \{ \rho, \sigma \}$, we have $(\alpha_i, \gamma) = (\alpha_i, \delta) = \emptyset$.
			
			\item $\epsilon' \subseteq \alpha_i$: For $R \in \partial^2 \epsilon \cap \partial^2 \rho = \partial^2 \epsilon' \cap \partial^2 \rho = \partial^2 \alpha \cap \partial^2 \beta$, we deduce $\emptyset \neq R\cap \epsilon' \subseteq \alpha_i$ and, as $R \notin \partial^2 \alpha_i$, we have $R \subseteq \alpha_i$. This yields $\beta_i = \alpha_i$. For $R \in \partial^2 \alpha_i \cap \partial^2 \epsilon$ we have $\emptyset \neq R \cap (-\alpha_i) \subseteq (-\epsilon')$. As $\partial^2 \alpha_i \cap \partial^2 \epsilon \cap \partial^2 \epsilon' = \partial^2 \alpha_i \cap \partial^2 \alpha \cap \partial^2 \beta = \emptyset$, we deduce $R \notin \partial^2 \epsilon'$ and hence $R \subseteq (-\epsilon')$. In particular, we have $\emptyset \neq (-\epsilon) \cap R \subseteq (-\epsilon) \cap (-\epsilon') = (-\alpha) \cap (-\beta) \subseteq (-\rho)$. As $R \notin \partial^2 \rho$ ($\{ r_{\alpha_i}, r_{\epsilon}, r_{\rho} \}$ is a reflection triangle), we infer $R \subseteq (-\rho)$ and hence $\beta_{\rho} = -\rho$. Lemma \ref{reflectiontrianglechamber} implies $(\alpha_i, \rho) = \emptyset$. Recall that $\beta_{\sigma} \in (\beta_{\epsilon}, \beta_{\rho})$. Using Lemma \ref{Lemma: Triangle and infinite order}, we deduce $o(r_{\alpha_i} r_{\sigma}) = \infty, -\alpha_i = -\beta_i \subseteq \beta_{\sigma}$ and $(-\alpha_i, \beta_{\sigma}) = \emptyset$. As $\{\alpha_i, \sigma\} \in \mathcal{P}$, Lemma \ref{Lemma: nested infinite order} implies that $\{ \alpha_i, \sigma \}$ is nested and hence $\{ -\alpha_i, -\sigma \}$ is nested as well. As $-\alpha_i \subseteq \beta_{\sigma}$, we deduce $\beta_{\sigma} = -\sigma$ and hence $(\alpha_i, \sigma) = -(-\alpha_i, -\sigma) = \emptyset$. Since $\{ \gamma, \delta \} = (\alpha, \beta) = \{ \rho, \sigma \}$, we have $(\alpha_i, \gamma) = (\alpha_i, \delta) = \emptyset$.
		\end{enumerate}
		
		In both cases we compute the following, where $N_{\alpha_i, \epsilon}^G \in \{ M_{\alpha_i, \epsilon}^G, M_{\epsilon, \alpha_i}^G \}$ depends on which expression is defined:
		\[ \prod\nolimits_{\epsilon \in M_{\alpha, \beta}^G} \left( \prod\nolimits_{\omega \in N_{\alpha_i, \epsilon}^G} u_{\omega} \right) u_{\epsilon} = \prod\nolimits_{\epsilon \in M_{\alpha, \beta}^G} u_{\epsilon} \qedhere \]
\end{proof}

\begin{definition}
	Let $H = (c_0, \ldots, c_k)$ be a gallery in $\Sigma(W, S)$. Then $H$ is called \textit{of type $(n, r) \in \NN_{\geq 1} \times S$}, if $S = \{r, s, t\}$ and the gallery $H$ is of type $(u, r, r_{\{s, t\}}, \ldots, r, r_{\{s, t\}}, r, v)$ for some $u, v \in \{ 1_W, s, t \}$, where $r_{\{s, t\}}$ appears $n$ times in the type of $H$. We note that $(1_W, c_0^{-1} c_1, \ldots, c_0^{-1} c_k)$ is a minimal gallery by Lemma \ref{wordsincoxetergroup} and \cite[Lemma $2.15$]{AB08} and so is $H$.
\end{definition}

\begin{lemma}\label{Lemma: extending gallery to gallery in Min}
	Let $\alpha, \beta \in \Phi_+$ be two roots and let $(n, r) \in \NN_{\geq 2} \times S$. Suppose that there exists a minimal gallery $H = (c_0, \ldots, c_k)$ of type $(n, r)$ between $\alpha$ and $\beta$. Then the following hold:
	\begin{enumerate}[label=(\alph*)]
		\item We can extend $(c_6, \ldots, c_k)$ to a minimal gallery contained in $\mathrm{Min}$.
		
		\item Let $R \in \partial^2 \alpha$ be a residue such that $\alpha$ is a non-simple root of $R$. If $\{ c_0, c_1 \} \not\subseteq R$, then there exists a simple root of $R$, say $\gamma \in \Phi_+$, such that $-\gamma \subseteq \beta$.
	\end{enumerate}
\end{lemma}
\begin{proof}
	In the proof we use the following notation: For a minimal gallery $G = (c_0, \ldots, c_k)$, $k \geq 1,$ we denote the unique root containing $c_{k-1}$ but not $c_k$ by $\alpha_G$.
	
	Let $S = \{r, s, t\}$ and recall that $H$ is of type $(u, r, r_{\{s, t\}}, \ldots, r, r_{\{s, t\}}, r, v)$, where $u, v \in \{1_W, s, t\}$ and $r_{\{s, t\}}$ appears $n$ times. We first show $(a)$. Suppose $u=1_W$. Then $\ell(c_0r) = \ell(c_0) +1$. If $\ell(c_0rs) = \ell(c_0r) +1 = \ell(c_0rt)$, we can extend $H$ to a gallery contained in $\mathrm{Min}$ by Lemma \ref{wordsincoxetergroup} and induction. If $\ell(c_0rs) = \ell(c_0)$, it follows from Lemma \ref{Lemma: residues in pairs} that $\ell(c_0rst) = \ell(c_0) +1$. Hence we can extend $(c_5, \ldots, c_k)$ to a gallery contained in $\mathrm{Min}$. The same holds if $\ell(c_0rt) = \ell(c_0)$. Now we suppose $u = s$ (the case $u=t$ is analogous). Again we note that $\ell(c_0s) = \ell(c_0)+1$. If $\ell(c_0sr) = \ell(c_0)$, Lemma \ref{Lemma: residues in pairs} yields $\ell(c_0srt) = \ell(c_0) +1$. Thus we can extend $(c_6, \ldots, c_k)$ to a gallery contained in $\mathrm{Min}$. Suppose that $\ell(c_0sr) = \ell(c_0)+2$. Note that Lemma \ref{wordsincoxetergroup} implies that $\ell(c_0srt) = \ell(c_0) +3$. If $\ell(c_0 srs) > \ell(c_0 sr)$, then we can extend $H$ to a gallery contained in $\mathrm{Min}$. Otherwise, Lemma \ref{Lemma: residues in pairs} again implies $\ell(c_0srst) = \ell(c_0) +2$ and we can extend $(c_6, \ldots, c_k)$ to a gallery contained in $\mathrm{Min}$. In any case we can extend $(c_6, \ldots, c_k)$ to a gallery $\Gamma \in \mathrm{Min}$. This proves $(a)$.
	
	To prove assertion $(b)$, we suppose $\{ c_0, c_1 \} \not\subseteq R$. As $P_{\alpha} \subseteq R$ by Lemma \ref{Lemma:Palpha}, we have $P_{\alpha} \neq \{c_0, c_1\}$. Let $P_0 = P_{\alpha}, \ldots, P_n = \{c_0, c_1\}$ and $R_1, \ldots, R_n$ be as in Lemma \ref{CM06Prop2.7}. For every $1 \leq i \leq n$ we define $w_i := \proj_{R_i} 1_W$, we let $\{x, y\}$ be the type of $R_n$, we let $\{x\}$ be the type of $\{c_0, c_1\}$ and we let $S = \{x, y, z\}$. We distinguish the following two cases:
	\begin{enumerate}[label=(\roman*)]
		\item\label{Case i} $\proj_{R_n} 1_W = \proj_{P_{n-1}} 1_W$: Depending on $H$, we have $\alpha_K \subseteq \beta$ by Lemma \ref{mingallinrep} for $K = (w_n, \ldots, n)$, where the type of $K$ is contained in $\{ (x, y, z, x), (x, y, x, z, y), (x, y, x, y, z) \}$.
		
		\item $\proj_{R_n} 1_W \neq \proj_{P_{n-1}} 1_W$: Depending on $H$, we have $\alpha_K \subseteq \beta$ by Lemma \ref{mingallinrep} for $K = (w_n, \ldots, w)$, where the type of $K$ is contained in $\{ (x, y, x, y, z), (x, y, x, z), (y, x, y, z, x) \}$.
	\end{enumerate}
	We note that for $L = (w_n, w_n x, w_n xy)$ we have $\alpha_L \subseteq \alpha_K$ by Lemma \ref{mingallinrep}, where $K$ is as in $(i)$ or $(ii)$. We distinguish the following cases:	
	\begin{enumerate}[label=(\alph*)]
		\item $R = R_1$: Then we have $n\geq 2$ (as $P_n \not\subseteq R$) and $\proj_{R_n} 1_W = \proj_{P_{n-1}} 1_W$ by Lemma \ref{corbasisroot}. Let $\gamma\in \Phi_+$ be the simple root of $R$ which does not contain $P_{\alpha}$. We first suppose $n=2$. Using Lemma \ref{mingallinrep} we deduce that $-\gamma$ is contained in all three roots $\alpha_K$ mentioned in \ref{Case i}. Note also that $-\gamma$ is contained in one non-simple root of $R_2$. Now we assume $n\geq 3$. Using induction, $-\gamma$ is contained in a non-simple root of $R_{n-1}$. As such a root is contained in both non-simple roots of $R_n$ by Lemma \ref{mingallinrep}, it follows $(-\gamma) \subseteq \alpha_L$, where $L = (w_n, w_nx, w_n xy)$ is as before.
		
		\item $R \neq R_1$: Let $\gamma \in \Phi_+$ be the simple root of $R$ containing $P_{\alpha}$. We prove by induction on $n$ that $(-\gamma) \subseteq \alpha_L$ holds, where $L$ is as before. We first suppose $n=1$. If $\proj_{R_1} 1_W \neq \proj_{P_0} 1_W$, then the claim follows from by Lemma \ref{mingallinrep}. Thus we can assume that $\proj_{R_1} 1_W = \proj_{P_0} 1_W$. As before, the claim follows from Lemma \ref{mingallinrep}. Now we suppose $n>1$. Using induction, $-\gamma$ is contained in a non-simple root of $R_{n-1}$. As such a root is contained both non-simple roots of $R_n$ by Lemma \ref{mingallinrep}, the claim follows. \qedhere
	\end{enumerate}
\end{proof}

\begin{lemma}\label{Lemma: unique (n,r)}
	Let $\alpha, \beta \in \Phi$ be two roots. Let $G = (c_0, \ldots, c_k)$ be a gallery of type $(n, r)$, let $H = (d_0, \ldots, d_l)$ be a gallery of type $(n', r')$ and suppose that $G$ and $H$ are minimal galleries between the same roots $\alpha$ and $\beta$. Then we have $n = n'$ and $r=r'$.
\end{lemma}
\begin{proof}
	Let $S = \{r, s, t\} = \{r', s', t'\}$. Recall that $G$ is of type $(u, r, r_{\{s, t\}}, \ldots, r, v)$ with $u, v \in \{1_W, s, t\}$ and $H$ is of type $(u', r', r_{\{s', t'\}}, \ldots, r', v')$ with $u', v' \in \{1_W, s', t'\}$. We first show that the last two chambers of the galleries $G$ and $H$ coincide, and that $r=r'$ and $v = v'$ hold. Using Lemma \ref{Lemma: extending gallery to gallery in Min}$(a)$ we can extend $(c_6, \ldots, c_k)$ and $(d_6, \ldots, d_l)$ to minimal galleries contained in $\mathrm{Min}$. We distinguish the following cases:
	\begin{enumerate}[label=(\alph*)]
		\item $v=1_W$: Then there ar exactly two different residues $R_1, R_2 \in \partial^2 \beta$ such that $\beta$ is not a simple root of $R_1$ and of $R_2$, and we have $\{ c_{k-1}, c_k\} = R_1 \cap R_2$. This implies $v' = 1_W$, as otherwise there would only be one such residue. But then $\{ c_{k-1}, c_k \} = P_{\beta} = \{ d_{l-1}, d_l \}$ and hence $r = r'$.
		
		\item $v = s$ (the case $v=t$ is analogous): Then, by Lemma \ref{wordsincoxetergroup} and \ref{Lemma: not both down}, there is only one residue $R \in \partial^2 \beta$ such that $\beta$ is not a simple root of $R$. Note that $c_{k-1}, c_k$ are contained in $R$ and we have $P_{\beta} \neq \{ c_{k-1}, c_k \}$. By the above we have $v' \neq 1_W$ and we deduce $P_{\beta} \neq \{ d_{l-1}, d_l \} \subseteq R$ similarly. This implies $v' = s$ and, in particular, $d_{l-1} = c_{k-1}$ and $d_l = c_k$. Note that there is only one element in $S$ which decreases the length of $d_{l-1} = c_{k-1}$. As $r, r'$ decrease both the length, we conclude $r=r'$.
	\end{enumerate}
	Assume $n \neq n'$. Without loss of generality we can assume $n < n'$. We consider the minimal galleries $H^{-1} = (d_l, d_{l-1}, \ldots, d_1, d_0)$ and $G^{-1} = (c_k, c_{k-1}, \ldots, c_1, c_0)$. Then $G^{-1}$ is a subgallery of $H^{-1}$ because of the types. But then $H$ crosses the wall $\partial \alpha$ at least twice. This is a contradiction to the fact that $H$ is minimal. We conclude $n=n'$ and the claim follows.
\end{proof}

In the next definition we will define subsets $M(n, r, L)_{\alpha, \beta}^G \subseteq (\alpha, \beta)$ for all $(G, \alpha, \beta) \in \mathcal{I}$, where $n \in \NN$ with $n \geq 3$, $r\in S$ and $L \subseteq \{2, \ldots, n-1 \}$. To have an intuition in mind, we will describe these symbols here: $n$ and $r$ will mean that there exists a minimal gallery of type $(n, r)$ between $\alpha$ and $\beta$; the subset $L$ roughly indicates, which elements of $(\alpha, \beta)$ are contained in the set $M(n, r, L)_{\alpha, \beta}^G$.

\begin{definition}\label{Definition: M(n,J)}
	\begin{enumerate}[label=(\alph*)]
		\item Let $r\in S$, let $n \in \NN$ with $n\geq 3$, let $L \subseteq \{ 2, \ldots, n-1 \}$ and let $(G, \alpha, \beta) \in \mathcal{I}$. If $o(r_{\alpha} r_{\beta}) < \infty$, then we define
		\[ M(n, r, L)_{\alpha, \beta}^G := \begin{cases*}
			(\alpha, \beta) & if $\vert (\alpha, \beta) \vert = 2$; \\
			\emptyset & else.
		\end{cases*} \]
		Now we consider the case $o(r_{\alpha} r_{\beta}) = \infty$. If there exists no minimal gallery of type $(n, r)$ between $\alpha$ and $\beta$, we define $M(n, r, L)_{\alpha, \beta}^G := \emptyset$. Now suppose that there exists a minimal gallery $H = (c_0, \ldots, c_k)$ of type $(n, r)$ between $\alpha$ and $\beta$. Let $(\alpha_1 = \alpha, \ldots, \alpha_k = \beta)$ be the sequence of roots crossed by $H$. For each $1 \leq i \leq n$ we define $\omega_i := \alpha_{k_i +2}$ and $\omega_i' := \alpha_{k_i +3}$, where $k_i = \ell\left( ur(r_{\{s, t\}}r)^{i-1} \right)$. We define 
		\[ M(n, r, L)_{\alpha, \beta}^G := \{ \omega_i, \omega_i' \mid i \in L \} \]
		
		\item Let $K \subseteq \NN_{\geq 3}$ be non-empty, let $\mathcal{J} = (J_k)_{k\in K}$ be a family of non-empty subsets $J_k \subseteq S$ and let $\mathcal{L} = \left(L_k^j\right)_{k \in K, j \in J_k}$ be a family of subsets $L_k^j \subseteq \{ 2, \ldots, k-1 \}$. For $(G, \alpha, \beta) \in \mathcal{I}$ we define
		\[ M(K, \mathcal{J}, \mathcal{L})_{\alpha, \beta}^G := \bigcup\nolimits_{k\in K, j \in J_k} M\left( k, j, L_k^j \right)_{\alpha, \beta}^G. \]
		Moreover, we let $\mathcal{M}(K, \mathcal{J}, \mathcal{L}) := \left( M(K, \mathcal{J}, \mathcal{L})_{\alpha, \beta}^G \right)_{(G, \alpha, \beta) \in \mathcal{I}}$.
	\end{enumerate}
\end{definition}

\begin{convention}
	For the rest of this section we let $K \subseteq \NN_{\geq 3}$ be non-empty, $\mathcal{J} = (J_k)_{k\in K}$ be a family of non-empty subsets $J_k \subseteq S$ and $\mathcal{L} = \left( L_k^j \right)_{k \in K, j \in J_k}$ be a family of subsets $L_k^j \subseteq \{ 2, \ldots, k-1 \}$. 
\end{convention}

\begin{remark}
	\begin{enumerate}[label=(\alph*)]
		\item Let $H$ be a gallery of type $(n, r)$ for some $n \in K$ and $r\in J_n$. We note that for $i>1$ the roots $\omega_i, \omega_i'$ are the non-simple roots of $R_{\{s, t\}}(c_{k_i})$. Suppose that $H$ is between $\alpha$ and $\beta$. Using Lemma \ref{mingallinrep}, we deduce $\alpha \subsetneq \omega_1, \omega_1' \subsetneq \cdots \subsetneq \omega_n, \omega_n' \subsetneq \beta$ and hence $M(n, r, L)_{\alpha, \beta}^G \subseteq (\alpha, \beta)$. We should remark that if $\alpha, \beta \in \Phi_+$, then not all of the roots crossed by $H$ are necessarily positive roots. But the roots $\omega_i, \omega_i'$ are. Consider for example the case $c_0 = trt$ and $H$ is of type $(r, r_{\{s, t\}}, r, \ldots, r_{\{s, t\}}, r)$. 
		
		\item We note that Lemma \ref{Lemma: unique (n,r)} implies that if $M(K, \mathcal{J}, \mathcal{L})_{\alpha, \beta}^G \neq \emptyset$, then there exist $k\in K$ and $j \in J_k$ with $M(K, \mathcal{J}, \mathcal{L})_{\alpha, \beta}^G = M\left( k, j, L_k^j \right)_{\alpha, \beta}^G$.
	\end{enumerate}
\end{remark}

\begin{lemma}
	 $\mathcal{M}(K, \mathcal{J}, \mathcal{L})$ is a pre-commutator blueprint of type $(4, 4, 4)$.
\end{lemma}
\begin{proof}
	We abbreviate $M_{\alpha, \beta}^G := M(K, \mathcal{J}, \mathcal{L})_{\alpha, \beta}^G$ for all $(G, \alpha, \beta) \in \mathcal{I}$. We first note that the sets $M_{\alpha, \beta}^G$ do not depend on $G$: in the case $o(r_{\alpha} r_{\beta}) < \infty$, the set $M_{\alpha, \beta}^G$ only depends on $\vert (\alpha, \beta) \vert$; in the case $o(r_{\alpha} r_{\beta}) = \infty$, the set $M_{\alpha, \beta}^G$ only depends on the existence of a suitable minimal gallery between $\alpha$ and $\beta$. On the other hand the order $\leq_G$ on the set $M_{\alpha, \beta}^G$ depends on $G$. Note that $\omega_i, \omega_i' \subsetneq \omega_{i+1}, \omega_{i+1}'$ and hence $\omega_i, \omega_i' \leq_G \omega_{i+1}, \omega_{i+1}'$, but the order on $\{ \omega_i, \omega_i' \}$ depends on $G$.
	
	Clearly, (CB$1$) and (CB$2$) hold. To show that (PCB) holds, we let $w\in W$ and $G \in \mathrm{Min}(w)$. Then we have a homomorphism $U_G \to U_w$ and it suffices to show that we have a homomorphism $U_w \to U_G$ extending $u_{\alpha} \mapsto u_{\alpha}$. Let $F \in \mathrm{Min}(w)$ and let $(F, \alpha, \beta) \in \mathcal{I}$. We first assume $o(r_{\alpha} r_{\beta}) < \infty$. If $\alpha \leq_G \beta$, then $M_{\alpha, \beta}^F = M_{\alpha, \beta}^G$ by definition and we are done. Thus we can assume $\beta \leq_G \alpha$. If $\vert (\alpha, \beta) \vert <2$, then $M_{\alpha, \beta}^F = \emptyset = M_{\beta, \alpha}^G$ and we are done. Thus we assume $(\alpha, \beta) = \{ \gamma, \delta \}$ and $\gamma \leq_F \delta$. Then $\delta \leq_G \gamma$ and we have the following relation in $U_G$:
	\[ [u_{\alpha}, u_{\beta}] = [u_{\beta}, u_{\alpha}]^{-1} = \left( u_{\delta} u_{\gamma} \right)^{-1} = u_{\gamma} u_{\delta} \]
	Thus we can consider the case $o(r_{\alpha} r_{\beta}) = \infty$. Then we have $\alpha \leq_G \beta$. If there is no gallery $H$ of type $(n, r)$ between $\alpha$ and $\beta$ for all $n\in K$ and $r\in J_n$, then $M_{\alpha, \beta}^F = \emptyset = M_{\alpha, \beta}^G$. Suppose that there exists a gallery $H$ of type $(n, r)$ between $\alpha$ and $\beta$ for some $n\in K$ and $r\in J_n$. Then $M_{\alpha, \beta}^F = \{ \omega_i, \omega_i' \mid i \in L_n^r \} = M_{\alpha, \beta}^G$ as sets. Note that $\omega_i, \omega_i' \leq_G \omega_{i+1}, \omega_{i+1}' \geq_F \omega_i, \omega_i'$. As $M_{\omega_i, \omega_i'}^G = \emptyset$, we deduce that $[u_{\alpha}, u_{\beta}] = \prod\nolimits_{\gamma \in M_{\alpha, \beta}^G} u_{\gamma} = \prod\nolimits_{\gamma \in M_{\alpha, \beta}^F} u_{\gamma}$ is a relation in $U_G$. Thus we obtain a homomorphism $U_w \to U_G$ and $\mathcal{M}(K, \mathcal{J}, \mathcal{L})$ is a pre-commutator blueprint.
\end{proof}

\begin{lemma}\label{Lemma: M(K,J,L) infinite order coincide}
	Let $(G, \alpha, \beta) \in \mathcal{I}$ with $o(r_{\alpha} r_{\beta}) < \infty$ and suppose that $M(K, \mathcal{J}, \mathcal{L})_{\alpha, \beta}^G = \{ \gamma, \delta \}$. Then the following hold for all $\epsilon \in \Phi(G)$:
	\begin{enumerate}[label=(\roman*)]
		\item If $\epsilon \subsetneq \gamma$ and $\epsilon \subsetneq \delta$, then we have $M(K, \mathcal{J}, \mathcal{L})_{\epsilon, \gamma}^G = M(K, \mathcal{J}, \mathcal{L})_{\epsilon, \delta}^G$.
		
		\item If $\gamma \subsetneq \epsilon$ and $\delta \subsetneq \epsilon$, then we have $M(K, \mathcal{J}, \mathcal{L})_{\gamma, \epsilon}^G = M(K, \mathcal{J}, \mathcal{L})_{\delta, \epsilon}^G$.
	\end{enumerate}
\end{lemma}
\begin{proof}
	We abbreviate $M_{\alpha, \beta}^G := M(K, \mathcal{J}, \mathcal{L})_{\alpha, \beta}^G$ for all $(G, \alpha, \beta) \in \mathcal{I}$. Let $R$ be the unique residue of rank $2$ contained in $\partial^2 \alpha \cap \partial^2 \beta$ (cf.\ Lemma \ref{Lemma: residues in the 2-boundary are parallel}$(b)$) and let $(\alpha, \beta) = \{ \gamma, \delta \}$. Suppose first $\epsilon \subsetneq \gamma$ and $\epsilon \subsetneq \delta$. If $M_{\epsilon, \gamma}^G = \emptyset = M_{\epsilon, \delta}^G$, we are done. Thus we can assume $M_{\epsilon, \gamma}^G \neq \emptyset$. Then there exists a minimal gallery $H = (d_0, \ldots, d_k)$ of type $(n, r)$ between the roots $\epsilon$ and $\gamma$ for some $n\in K$ and $r\in J_n$. In particular, the type of $H$ is given by $(u, r, r_{\{s, t\}}, r, \ldots, r_{\{s, t\}}, r, v)$, where $u, v \in \{1_W, s, t\}$ and $r_{\{s, t\}}$ appears $n$ times. Using Lemma \ref{Lemma: extending gallery to gallery in Min}$(a)$, we can extend $(d_6, \ldots, d_k)$ to a gallery $\Gamma \in \mathrm{Min}$. Using Lemma \ref{wordsincoxetergroup} and induction, $\gamma$ is a non-simple root of the residue of rank $2$ containing $d_{k-2}, d_{k-1}$ and $d_k$. We distinguish the following cases:
	\begin{enumerate}[label=(\roman*)]
		\item $v=1$: Then we have $P_{\gamma} = \{ d_{k-1}, d_k \}$ (cf.\ Lemma \ref{Lemma:Palpha}) and $P_{\gamma} \subseteq R$. There exists $x\in \{s, t\}$ such that the minimal gallery $H$ extended by an $x$-adjacent chamber is still of type $(n, r)$ and between the roots $\epsilon$ and $\delta$. We deduce $M_{\epsilon, \gamma}^G = M_{\epsilon, \delta}^G$.
		
		\item $v\neq 1$: Using Lemma \ref{Lemma: not both down} and Lemma \ref{wordsincoxetergroup} we deduce that $R$ is the only residue such that $\gamma$ is a non-simple root of $R$. Then the minimal gallery $K := (d_0, \ldots, d_{k-1})$ is still of type $(n, r)$ and between the roots $\epsilon$ and $\delta$. We deduce $M_{\epsilon, \gamma}^G = M_{\epsilon, \delta}^G$.
	\end{enumerate}
	
	Now we assume $\gamma \subsetneq \epsilon$ and $\delta \subsetneq \epsilon$. If $M_{\gamma, \epsilon}^G = \emptyset = M_{\delta, \epsilon}^G$, we are done. Thus we can assume $M_{\gamma, \epsilon}^G \neq \emptyset$. Then there exists a minimal gallery $H = (d_0, \ldots, d_k)$ of type $(n, r)$ between $\gamma$ and $\epsilon$ for some $n\in K$ and $r \in J_n$. In particular, the type of $H$ is given by $(u, r, r_{\{s, t\}}, r, \ldots, r_{\{s, t\}}, r, v)$, where $u, v \in \{ 1_W, s, t \}$ and $r_{\{s, t\}}$ appears $n$ times. Clearly, $\gamma$ is a non-simple root of $R$. Note that $\alpha, \beta, \epsilon \in \Phi(G)$ and hence $\{ \alpha, \epsilon \}, \{ \beta, \epsilon \} \in \P$. Then Lemma \ref{Lemma: extending gallery to gallery in Min}$(b)$ yields $\{ d_0, d_1 \} \subseteq R$. We distinguish the following two cases:
	\begin{enumerate}[label=(\roman*)]
		\item $u=1_W$: Assume that $\ell(\proj_R 1_W, d_0) = 1$. Then Lemma \ref{mingallinrep} would imply that one of $-\alpha, -\beta$ is contained in one non-simple root of the $\{s, t\}$-residue containing $d_1$ (i.e.\ $\omega_1$ or $\omega_1'$) and hence one of $-\alpha, -\beta$ is contained in $\epsilon$. As this is a contradiction, we deduce $\ell(\proj_R 1_W, d_0) = 2$. Let $d$ be the chamber in $R$ adjacent to both $\proj_R 1_W$ and $d_0$. Then the gallery $(d, d_0, \ldots, d_k)$ is of type $(n, r)$ and between the roots $\delta$ and $\epsilon$. We deduce $M_{\gamma, \epsilon}^G = M_{\delta, \epsilon}^G$.
		
		\item $u\neq 1_W$: Assume $\ell(\proj_R 1_W, d_0) = 2$. In both cases ($d_2 \in R$ and $d_2 \notin R$) Lemma \ref{mingallinrep} would imply that one of $-\alpha, -\beta$ is contained in $\omega_2, \omega_2'$ and hence in $\epsilon$, which is a contradiction. Thus $\ell(\proj_R 1_W, d_0) = 1$. Again, if $d_2 \notin R$, then Lemma \ref{mingallinrep} would imply that one of $-\alpha, -\beta$ is contained in $\omega_1, \omega_1'$, which is a contradiction. Thus we can assume $d_2 \in R$. As $(d_1, \ldots, d_k)$ is still a gallery of type $(n, r)$ between the roots $\delta$ and $\epsilon$, we deduce $M_{\gamma, \epsilon}^G = M_{\delta, \epsilon}^G$ and the claim follows. \qedhere
	\end{enumerate}
\end{proof}

\begin{lemma}\label{Lemma: 2-nilpotent prec-ommutator blueprint}
	$\mathcal{M}(K, \mathcal{J}, \mathcal{L})$ is a Weyl-invariant commutator blueprint and the groups $U_w$ are nilpotent of class at most $2$.
\end{lemma}
\begin{proof}
	We abbreviate $M_{\alpha, \beta}^G := M(K, \mathcal{J}, \mathcal{L})_{\alpha, \beta}^G$ for all $(G, \alpha, \beta) \in \mathcal{I}$. We will apply Proposition \ref{Proposition: 2-nilpotent}. Suppose $(G, \alpha, \beta) \in \mathcal{I}$ with $o(r_{\alpha} r_{\beta}) < \infty$. Then we have
	\[ M_{\alpha, \beta}^G = \begin{cases}
		(\alpha, \beta) & \vert (\alpha, \beta) \vert = 2 \\
		\emptyset & \vert (\alpha, \beta) \vert < 2
	\end{cases} \]
	by definition. Let $G = (c_0, \ldots, c_k) \in \mathrm{Min}$, let $(\alpha_1, \ldots, \alpha_k)$ be the sequence of roots crossed by $G$, let $1 \leq j \leq k-1$ and let $\alpha := \alpha_j, \beta := \alpha_k$. Then condition $(a)$ of Proposition \ref{Proposition: 2-nilpotent} follows from Lemma \ref{Lemma: M(K,J,L) infinite order coincide}. Now we will show that condition $(b)$ holds. For that we suppose $o(r_{\alpha} r_{\beta}) = \infty$. If $M_{\alpha, \beta}^G = \emptyset$, we are done. Thus we can assume $M_{\alpha, \beta}^G \neq \emptyset$. Then there exists a minimal gallery of type $(n, r)$ between $\alpha$ and $\beta$ for some $n \in K$ and $r\in J_n$. In particular, we have $M_{\alpha, \beta}^G = \{ \omega_i, \omega_i' \mid i \in L_n^r \}$. Note that $\{ \omega_i, \omega_i' \} = M_{\gamma_i, \gamma_i'}^G$ for some $\gamma_i \leq_G \gamma_i' \in \Phi(G)$ with $\alpha \subseteq \gamma_i \leq_G \gamma_i' \subseteq \beta$, as $L_n^r \subseteq \{ 2, \ldots, n-1 \}$. We remark that we really need $1, n \notin L_n^r$. 
	
	We first show that $u_{\omega_i} u_{\omega_i'} \in Z(U_{(c_0, \ldots, c_{k-1})})$ for all $i \in L_n^r$. This implies that $\prod\nolimits_{\gamma \in M_{\alpha, \beta}^G} u_{\gamma} = \prod\nolimits_{i \in L_n^r} u_{\omega_i} u_{\omega_i'}$ is contained in $Z(U_{(c_0, \ldots, c_{k-1})})$ and, moreover, we deduce
	\[ \left( \prod\nolimits_{\gamma \in M_{\alpha, \beta}^G} u_{\gamma} \right)^2 = \left( \prod\nolimits_{i \in L_n^r} u_{\omega_i} u_{\omega_i'} \right)^2 = \prod\nolimits_{i\in L_n^r} (u_{\omega_i} u_{\omega_i'})^2 = 1.\]
	Note that the order $\leq_G$ on $\{ \omega_i, \omega_i' \}$ depends on $G$. Let $\epsilon \in \Phi(G) \backslash \{ \beta \}$. Then we have to show that $u_{\omega_i} u_{\omega_i'}$ and $u_{\epsilon}$ commute in $U_{(c_0, \ldots, c_{k-1})}$. But this group is nilpotent of class at most $2$ by induction and $u_{\omega_i} u_{\omega_i'} \in [U_{(c_0, \ldots, c_{k-1})}, U_{(c_0, \ldots, c_{k-1})}] \leq Z(U_{(c_0, \ldots, c_{k-1})})$.
	
	Now it is left to show that ($2$-n$3$) of Theorem \ref{Theorem: pre-com = com}\ref{nilpotent conditions} is satisfied. Let $1 \leq i < j \leq k$. Note that by construction of the sets $M_{\alpha, \beta}^G$, it suffices to show the claim for the case $o(r_{\alpha_i} r_{\alpha_j}) < \infty$ and $M_{\alpha_i, \alpha_j}^G \neq \emptyset$. Suppose that $M_{\alpha_i, \alpha_j}^G = \{ \gamma, \delta \}$ with $\gamma \leq_G \delta$. We distinguish the following cases:
	\begin{enumerate}[label=(\alph*)]
		\item $o(r_{\gamma} r_{\alpha_k}) < \infty$: Assume that $\alpha_i \subseteq \alpha_k$ and $\alpha_j \subseteq \alpha_k$. Then $(-\alpha_k) \subseteq (-\alpha_i) \cap (-\alpha_j) \subseteq (-\gamma)$ would yield a contradiction. We deduce $\alpha_i \not\subseteq \alpha_k$ or $\alpha_j \not\subseteq \alpha_k$. Suppose $\epsilon \in \{ \alpha_i, \alpha_j \}$ with $o(r_{\epsilon} r_{\alpha_k}) < \infty$. If $\epsilon = \alpha_j = \alpha_k$, then the claim follows. Thus we can assume $\alpha_j \neq \alpha_k$. It follows $\partial^2 \epsilon \cap \partial^2 \gamma \cap \partial^2 \alpha_k = \partial^2 \alpha_i \cap \partial^2 \alpha_j \cap \partial^2 \alpha_k = \emptyset$ from Lemma \ref{Lemma: residue cut by intersection reflections} and hence$\{ r_{\epsilon}, r_{\gamma}, r_{\alpha_k} \}$ is a reflection triangle. By Remark \ref{Remark: reflection triangle plus orientation yields triangle} there exist $\beta_{\epsilon} \in \{ \epsilon, -\epsilon \}, \beta_{\gamma} \in \{ \gamma, -\gamma \}$ and $\beta_k \in \{ \alpha_k, -\alpha_k \}$ such that $\{ \beta_{\epsilon}, \beta_{\gamma}, \beta_k \}$ is a triangle.
		
		Let $\epsilon' \in \{ \alpha_i, \alpha_j \} \backslash \{\epsilon\}$. Then we know by Lemma \ref{Lemma: residue cut by intersection reflections} and \ref{Lemma: Triangle and infinite order}$(a)$ that $o(r_{\delta} r_{\alpha_k}) = \infty = o(r_{\epsilon'} r_{\alpha_k})$. Lemma \ref{Lemma: nested infinite order}$(a)$ together with the fact $\{ \delta, \alpha_k \}, \{ \epsilon', \alpha_k \} \in \mathcal{P}$ imply that both pairs are nested. As $\delta \leq_G \alpha_k$ and $\epsilon' \leq_G \alpha_k$, we deduce $\delta \subseteq \alpha_k$ and $\epsilon' \subseteq \alpha_k$. Let $R \in \partial^2 \epsilon \cap \partial^2 \gamma = \partial^2 \alpha_i \cap \partial^2 \alpha_j \cap \partial^2 \delta$. Then $\emptyset \neq R \cap \delta \subseteq \alpha_k$ and (as $R \notin \partial^2 \alpha_k$) we deduce $R \subseteq \alpha_k$. This implies $\beta_k = \alpha_k$.
		
		Note that there exist $\beta_{\delta} \in \{ \delta, -\delta \}$ and $\beta_{\epsilon'} \in \{ \epsilon', -\epsilon' \}$ with $\beta_{\delta}, \beta_{\epsilon'} \in (\beta_{\gamma}, \beta_{\epsilon})$ by Lemma \ref{Lemma: residue cut by intersection reflections} and \ref{reflectiontrianglechamber}. Using Lemma \ref{Lemma: Triangle and infinite order}$(a)$ we obtain $(-\alpha_k) \subseteq \beta_{\delta}$ and $(-\alpha_k) \subseteq \beta_{\epsilon'}$, and hence $(-\delta) = \beta_{\delta}, (-\epsilon') = \beta_{\epsilon'} \in (\beta_{\gamma}, \beta_{\epsilon})$. By Lemma \ref{Lemma: interval of roots} we have $(-\epsilon') \in (\epsilon, -\gamma)$ and hence $\beta_{\gamma} = -\gamma, \beta_{\epsilon} = \epsilon$. Lemma \ref{Lemma: Triangle and infinite order}$(b)$ implies $(\delta, \alpha_k) = -(-\delta, -\alpha_k) = \emptyset$. Moreover, Lemma \ref{reflectiontrianglechamber} yields $(\gamma, \alpha_k) = (-\beta_{\gamma}, \beta_k) = \emptyset$ and hence $M_{\gamma, \alpha_k}^G = \emptyset = M_{\delta, \alpha_k}^G$. We conclude
		\allowdisplaybreaks
		\begin{align*}
			\prod\nolimits_{\sigma \in M_{\alpha_i, \alpha_j}^G} \left( u_{\sigma} \prod\nolimits_{\rho \in M_{\sigma, \alpha_k}^G} u_{\rho} \right) &=	\left( u_{\gamma} \prod\nolimits_{\rho \in M_{\gamma, \alpha_k}^G} u_{\rho} \right) \left( u_{\delta} \prod\nolimits_{\rho \in M_{\delta, \alpha_k}^G} u_{\rho} \right) \\
			&= u_{\gamma} u_{\delta} \\
			&= \prod\nolimits_{\sigma \in M_{\alpha_i, \alpha_j}^G} u_{\sigma}
		\end{align*}
				
		\item $o(r_{\gamma} r_{\alpha_k}) = \infty = o(r_{\delta} r_{\alpha_k})$: It follows from Lemma \ref{Lemma: M(K,J,L) infinite order coincide} that $M_{\gamma, \alpha_k}^G = M_{\delta, \alpha_k}^G$. We have already shown that condition $(b)$ of Proposition \ref{Proposition: 2-nilpotent} holds. Thus we know that $\prod\nolimits_{\epsilon \in M_{\gamma, \alpha_k}^G} u_{\epsilon} \in Z(U_{(c_0, \ldots, c_{k-1})})$ as well as $\left( \prod\nolimits_{\epsilon \in M_{\gamma, \alpha_k}^G} u_{\epsilon} \right)^2 = 1$ in the group $U_{(c_0, \ldots, c_{k-1})}$. We deduce the following:
		\allowdisplaybreaks
		\begin{align*}
			\left( u_{\gamma} \prod\nolimits_{\rho \in M_{\gamma, \alpha_k}^G} u_{\rho} \right) \left( u_{\delta} \prod\nolimits_{\rho \in M_{\delta, \alpha_k}^G} u_{\rho} \right) &= \left( u_{\gamma} \prod\nolimits_{\rho \in M_{\gamma, \alpha_k}^G} u_{\rho} \right) \left( u_{\delta} \prod\nolimits_{\rho \in M_{\gamma, \alpha_k}^G} u_{\rho} \right) \\
			&= u_{\gamma} u_{\delta} \left( \prod\nolimits_{\rho \in M_{\gamma, \alpha_k}^G} u_{\rho} \right)^2 \\
			&= u_{\gamma} u_{\delta}
		\end{align*} 
	\end{enumerate}
	
	It is left to show that $\mathcal{M}(K, \mathcal{J}, \mathcal{L})$ is Weyl-invariant. Let $w\in W, s\in S, G \in \mathrm{Min}_s(w)$ and $(G, \alpha, \beta) \in \mathcal{I}$ with $\alpha \neq \alpha_s \neq \beta$. If $o(r_{\alpha} r_{\beta}) < \infty$, then $o(r_{s\alpha} r_{s\beta}) < \infty$ and, as $(s\alpha, s\beta) = \{ s\gamma \mid \gamma \in (\alpha, \beta) \}$, we infer $M_{s\alpha, s\beta}^{sG} = sM_{\alpha, \beta}^G$. Thus we can assume $o(r_{\alpha} r_{\beta}) = \infty$. Suppose that there exists a gallery $H = (c_0, \ldots, c_k)$ of type $(n, r)$ between $\alpha$ and $\beta$ for some $n\in K, r\in J_n$. Then $(sc_0, \ldots, sc_k)$ is a gallery of type $(n, r)$ between the roots $s\alpha, s\beta$. This implies that a gallery of type $(n, r)$ exists between the roots $\alpha$ and $\beta$ if and only if a gallery of type $(n, r)$ exists between the roots $s\alpha$ and $s\beta$. This finishes the claim.
\end{proof}

\bibliographystyle{amsalpha}
\bibliography{references}

\providecommand{\bysame}{\leavevmode\hbox to3em{\hrulefill}\thinspace}
\providecommand{\MR}{\relax\ifhmode\unskip\space\fi MR }
% \MRhref is called by the amsart/book/proc definition of \MR.
\providecommand{\MRhref}[2]{%
  \href{http://www.ams.org/mathscinet-getitem?mr=#1}{#2}
}
\providecommand{\href}[2]{#2}
\begin{thebibliography}{DMVM12}

\bibitem[AB08]{AB08}
Peter Abramenko and Kenneth~S. Brown, \emph{Buildings}, Graduate Texts in
  Mathematics, vol. 248, Springer, New York, 2008, Theory and applications.
  \MR{2439729}

\bibitem[Bisa]{BiCoxGrowth}
Sebastian Bischof, \emph{On {G}rowth {F}unctions of {C}oxeter {G}roups},
  manuscript, 11 pp.

\bibitem[Bisb]{BiRGD}
\bysame, \emph{{RGD}-systems over $\mathbb{F}_2$}, manuscript, 30 pp.

\bibitem[Bis22]{Bi22}
\bysame, \emph{On commutator relations in 2-spherical {RGD}-systems}, Comm.
  Algebra \textbf{50} (2022), no.~2, 751--769. \MR{4375537}

\bibitem[Bis23]{BiDiss}
\bysame, \emph{Construction of {RGD}-systems of type $(4,4,4)$ over
  $\mathbb{F}_2$}, PhD thesis, Justus-Liebig-Universität Giessen, 2023.

\bibitem[Bou02]{Bo68}
Nicolas Bourbaki, \emph{Lie groups and {L}ie algebras. {C}hapters 4--6},
  Elements of Mathematics (Berlin), Springer-Verlag, Berlin, 2002, Translated
  from the 1968 French original by Andrew Pressley. \MR{1890629}

\bibitem[Cap07]{Ca07}
Pierre-Emmanuel Caprace, \emph{A uniform bound on the nilpotency degree of
  certain subalgebras of {K}ac-{M}oody algebras}, J. Algebra \textbf{317}
  (2007), no.~2, 867--876. \MR{2362945}

\bibitem[CM05]{CM05}
Pierre-Emmanuel Caprace and Bernhard M\"{u}hlherr, \emph{Reflection triangles
  in {C}oxeter groups and biautomaticity}, J. Group Theory \textbf{8} (2005),
  no.~4, 467--489. \MR{2152693}

\bibitem[CM06]{CM06}
\bysame, \emph{Isomorphisms of {K}ac-{M}oody groups which preserve bounded
  subgroups}, Adv. Math. \textbf{206} (2006), no.~1, 250--278. \MR{2261755}

\bibitem[CR16]{CR16}
Pierre-Emmanuel Caprace and Bertrand R\'{e}my, \emph{Simplicity of twin tree
  lattices with non-trivial communication relations}, Topology and geometric
  group theory, Springer Proc. Math. Stat., vol. 184, Springer, [Cham], 2016,
  pp.~143--151. \MR{3598164}

\bibitem[DMVM12]{DMVM11}
Alice Devillers, Bernhard M\"{u}hlherr, and Hendrik Van~Maldeghem,
  \emph{Codistances of 3-spherical buildings}, Math. Ann. \textbf{354} (2012),
  no.~1, 297--329. \MR{2957628}

\bibitem[Fel98]{Fe98}
A.~A. Felikson, \emph{Coxeter decompositions of hyperbolic polygons}, European
  J. Combin. \textbf{19} (1998), no.~7, 801--817. \MR{1649962}

\bibitem[GHM]{GHM16preprint2}
Matthias Gr\"{u}ninger, Max Horn, and Bernhard M\"{u}hlherr, \emph{Moufang twin
  trees of prime order, part 2, manuscript, {G}iessen, 27pp.}

\bibitem[GHM16]{GHM16}
\bysame, \emph{Moufang twin trees of prime order}, Adv. Math. \textbf{302}
  (2016), 1--24. \MR{3545922}

\bibitem[M\"99]{Mu99}
Bernhard M\"{u}hlherr, \emph{Locally split and locally finite twin buildings of
  {$2$}-spherical type}, J. Reine Angew. Math. \textbf{511} (1999), 119--143.
  \MR{1695793}

\bibitem[Par21]{PaDiss21}
M.~Parr, \emph{Moufang twin trees and $\mathbb{Z}$-systems}, PhD thesis,
  Justus-Liebig-Universität Giessen, 2021.

\bibitem[R\'02]{Re03}
Bertrand R\'{e}my, \emph{Groupes de {K}ac-{M}oody d\'{e}ploy\'{e}s et presque
  d\'{e}ploy\'{e}s}, Ast\'{e}risque (2002), no.~277, viii+348. \MR{1909671}

\bibitem[RR06]{RR06}
Bertrand R\'{e}my and Mark Ronan, \emph{Topological groups of {K}ac-{M}oody
  type, right-angled twinnings and their lattices}, Comment. Math. Helv.
  \textbf{81} (2006), no.~1, 191--219. \MR{2208804}

\bibitem[Seg09]{Seg09}
Yoav Segev, \emph{Proper {M}oufang sets with abelian root groups are special},
  J. Amer. Math. Soc. \textbf{22} (2009), no.~3, 889--908. \MR{2505304}

\bibitem[SW08]{SW08}
Yoav Segev and Richard~M. Weiss, \emph{On the action of the {H}ua subgroups in
  special {M}oufang sets}, Math. Proc. Cambridge Philos. Soc. \textbf{144}
  (2008), no.~1, 77--84. \MR{2388234}

\bibitem[Tit87]{Ti87}
Jacques Tits, \emph{Uniqueness and presentation of {K}ac-{M}oody groups over
  fields}, J. Algebra \textbf{105} (1987), no.~2, 542--573. \MR{873684}

\bibitem[Tit92]{Ti92}
\bysame, \emph{Twin buildings and groups of {K}ac-{M}oody type}, Groups,
  combinatorics \& geometry ({D}urham, 1990), edited by M. Liebeck and J. Saxl,
  London Math. Soc. Lecture Note Ser., vol. 165, Cambridge Univ. Press,
  Cambridge, 1992, pp.~249--286. \MR{1200265}

\bibitem[Tit13]{Ti73-00}
\bysame, \emph{R\'{e}sum\'{e}s des cours au {C}oll\`ege de {F}rance
  1973--2000}, Documents Math\'{e}matiques (Paris), vol.~12, Soci\'{e}t\'{e}
  Math\'{e}matique de France, Paris, 2013. \MR{3235648}

\bibitem[Wei03]{We03}
Richard~M. Weiss, \emph{The structure of spherical buildings}, Princeton
  University Press, Princeton, NJ, 2003. \MR{2034361}

\bibitem[Wei09]{We09}
\bysame, \emph{The structure of affine buildings}, Annals of Mathematics
  Studies, vol. 168, Princeton University Press, Princeton, NJ, 2009.
  \MR{2468338}

\end{thebibliography}

\end{document}